\documentclass[12pt]{amsart}
\usepackage{etex}
\usepackage{lscape}
\usepackage[usenames,dvipsnames]{pstricks}
\usepackage{epsfig}
\usepackage{graphicx,color}
\usepackage{geometry}
\usepackage[all]{xy}
\usepackage{amssymb}
\usepackage{cite}
\usepackage{fullpage}
\xyoption{poly}
\usepackage{url}
\usepackage{algorithm}
\usepackage{hyperref}
\numberwithin{equation}{subsection}
\numberwithin{figure}{subsection}
\usepackage{longtable}
\usepackage{amssymb, verbatim, amsmath,amsthm,amsfonts}
\usepackage{color}
\usepackage{amsfonts}
\usepackage[all]{xy}
\usepackage{amsthm}
\usepackage{algpseudocode}
\usepackage{MnSymbol}
\usepackage{pifont}
\usepackage{ifsym}
\usepackage{graphicx} 
\DeclareMathAlphabet{\mathpgoth}{OT1}{pgoth}{m}{n}
\DeclareMathAlphabet{\mathesstixfrak}{U}{esstixfrak}{m}{n}
\DeclareMathAlphabet{\mathboondoxfrak}{U}{BOONDOX-frak}{m}{n}
\DeclareMathAlphabet{\mathtxfrak}{U}{tx-frak}{m}{n}
\DeclareMathAlphabet{\matheulerfrak}{U}{euf}{m}{n}

\newcommand{\finsub}[0]{{\subseteq}_{\mathsf{fin}}}

\newcommand{\bc}{\begin{center}}
\newcommand{\ec}{\end{center}}
\newcommand{\be}{\begin{enumerate}}
\newcommand{\ee}{\end{enumerate}}
\newcommand{\beq}{\begin{equation}}
\newcommand{\eeq}{\end{equation}}
\newcommand{\bi}{\begin{itemize}}
\newcommand{\ei}{\end{itemize}}
\newcommand{\bd}{\begin{description}}
\newcommand{\ed}{\end{description}}
\newcommand{\ba}{\begin{array}}
\newcommand{\bea}{\begin{eqnarray*}}
\newcommand{\eea}{\end{eqnarray*}}
\newcommand{\ea}{\end{array}}
\newcommand{\bt}{\begin{tabular}}
\newcommand{\et}{\end{tabular}}

\newcommand{\bmi}{\begin{minipage}}
\newcommand{\emi}{\end{minipage}}

\newcommand{\lb}{\linebreak}

\newcommand{\myfraku}[1]{\ensuremath{(F_u)_{\mathfrak{#1}}}}
\newcommand{\myfrak}[1]{\ensuremath{F_{\mathfrak{#1}}}}
\newcommand{\myfraklu}[1]{\ensuremath{(F_{\lambda u})_{\mathfrak{#1}}}}

\newcommand{\myfrakdu}[1]{\ensuremath{(F_{\delta_{j} u})_{\mathfrak{#1}}}}
\newcommand{\myfrakv}[1]{\ensuremath{(F_{v})_{\mathfrak{#1}}}}

\newcommand{\myfrakk}[1]{\ensuremath{(\mathbb{F}_{5^2 u })_{\mathfrak{#1}}}}

\newtheorem{stel}{Theorem}[section]
\newtheorem{defin}[stel]{Definition}
\newtheorem{definlem}[stel]{Definition-Lemma}
\newtheorem{lemm}[stel]{Lemma}

\newtheorem{cor}[stel]{Corollary}

\newtheorem{exam}[stel]{Example}
\newtheorem{rem}[stel]{Remark}
\newtheorem{corollary}[stel]{Corollary}
\newtheorem{theo}[stel]{Theorem}

\makeatletter
\newcommand{\myitem}[1]{%
\item[#1]\protected@edef\@currentlabel{#1}%
}
\makeatother
\usepackage{multicol}

\usepackage{MnSymbol,bbding,pifont}

\makeatletter
\newcommand{\Matrix}[1]
    {\begin{pmatrix}
      \Matrix@r #1;\@bye;\Matrix@r
     \end{pmatrix}}

\def\Matrix@r #1;{\@bye #1\Matrix@z\@bye\Matrix@s #1,\@bye, }%
\def\Matrix@s #1,{#1\Matrix@t }%
\def\Matrix@t #1,{\@bye #1\Matrix@y\@bye\@firstofone {&#1}\Matrix@t}%
\def\Matrix@y #1\Matrix@t{\\ \Matrix@r }%
\def\Matrix@z #1\Matrix@r {}
\def\@bye  #1\@bye   {}

\makeatother

\makeatletter
\newsavebox\myboxA
\newsavebox\myboxB
\newlength\mylenA

\newcommand*\xoverline[2][0.75]{%
    \sbox{\myboxA}{$\m@th#2$}%
    \setbox\myboxB\null
    \ht\myboxB=\ht\myboxA%
    \dp\myboxB=\dp\myboxA%
    \wd\myboxB=#1\wd\myboxA
    \sbox\myboxB{$\m@th\overline{\copy\myboxB}$}
    \setlength\mylenA{\the\wd\myboxA}
    \addtolength\mylenA{-\the\wd\myboxB}%
    \ifdim\wd\myboxB<\wd\myboxA%
       \rlap{\hskip 0.5\mylenA\usebox\myboxB}{\usebox\myboxA}%
    \else
        \hskip -0.5\mylenA\rlap{\usebox\myboxA}{\hskip 0.5\mylenA\usebox\myboxB}%
    \fi}
\makeatother

\makeatletter
\newcommand\ackname{Acknowledgements}
\if@titlepage
  \newenvironment{acknowledgements}{%
      \titlepage
      \null\vfil
      \@beginparpenalty\@lowpenalty
      \begin{center}%
        \bfseries \ackname
        \@endparpenalty\@M
      \end{center}}%
     {\par\vfil\null\endtitlepage}
\else
  
\fi
\makeatother
\title{\bfseries Distributive decomposition of near-vector spaces}
\author{Leandro Boonzaaier, Sophie Marques and Daniella Moore}

\begin{document}
\maketitle
\bc  

\it\small
Rain Networks (Pty) Ltd, 
Cape Quarter, 
Green Point, 8051,\lb
South Africa\\
\rm e-mail address: leandro.boonzaaier@rain.co.za

\it
Department of Mathematical Sciences, 
University of Stellenbosch, 
Stellenbosch, 7600,\lb
South Africa\\
\&
NITheCS (National Institute for Theoretical and Computational Sciences), 
South Africa \\
\rm e-mail: smarques@sun.ac.za

\it
Department of Mathematical Sciences, 
University of Stellenbosch, 
Stellenbosch, 7600,\lb
South Africa\\
\rm e-mail: dmoore@sun.ac.za

\ec

\begin{abstract} 
This paper provides two characterizations of regularity for near-vector spaces: first, by expressing them as a direct sum of vector spaces over division rings formed by distributive elements; second, by expressing their dimension in term of the dimension of these summands. These results offer new insights into the structure and properties of near-vector spaces.
\end{abstract} 

{\bf Key words:} Near-vector spaces, Near-rings, Near-fields, Quasi-kernel, Distributive elements, Division rings, Vector space decomposition.

{\it 2020 Mathematics Subject Classification:} 16Y30; 05B35 \bigskip

\tableofcontents

\section{Introduction}

In \cite{Andre}, J. André introduced the concept of a near-vector space, which generalizes traditional vector spaces by allowing the distributive law to hold on only one side. This generalization permits certain non-linear behaviors while still utilizing many linear algebra methods. Recent research has explored these structures from various perspectives, including algebraic, geometric, and categorical approaches (see \cite{DeBruyn, Howell, Howell2, HM22, HR22, HS18, MarquesMoore} for examples).

André demonstrated that any near-vector space can be decomposed into a direct sum of maximal regular near-vector spaces (see \cite[Satz 4.14]{Andre}). For a near-vector space over a scalar group \((F, \cdot, 1, 0, -1)\) (see Definition \ref{scalargroup}), he showed that any non-zero element \(u\) in the quasi-kernel (see Definition \ref{NVS2}) defines an addition \(+_u\) such that \((F, +_u, \cdot)\) forms a near-field (see \cite[Satz 2.4]{Andre}). When \((F, +_u, \cdot)\) is a division ring for all non-zero \(u\) in the quasi-kernel, it implies that a regular near-vector space is simply a vector space (see \cite[Theorem 4.2]{HM22}). However, when this condition is not met, a regular near-vector space exhibits more subtle characteristics.

In this paper, we characterize regularity using a direct sum of vector spaces, even when \((F, +_u, \cdot)\) is not a division ring for all non-zero \(u\) in the quasi-kernel. The key observation enabling this characterization is that the set of distributive elements of \((F, +_u, \cdot)\) forms a division ring for each non-zero \(u\) in the quasi-kernel, and there exists a canonical method to define a vector space over each of these division rings (see \cite[Theorem 2.4-3]{DeBruyn}). More precisely, we identify an equivalence relation that groups these division rings into classes of elements with equal additions, which in turn generates a vector space over these division rings.

Our main theorem (see Theorem \ref{thm1}) characterizes the regularity of near-vector spaces by demonstrating that they can be expressed as a direct sum of these vector spaces. Additionally, we establish that the dimension of a regular near-vector space is equal to the dimension of each of these direct summands over the corresponding division ring, which are vector spaces over the distributive elements of a near-field \((F, +_u, \cdot)\) for some non-zero \(u\) in the quasi-kernel (see Corollary \ref{dimension}). This dimensional property also serves as a criterion for regularity (see Theorem \ref{thm1}).

These results offer deeper insights into the structure and properties of near-vector spaces, particularly regarding the fact that the dimension of an element can be greater than 1. Additionally, we provide methods to compute this dimension. By relating near-linear algebra to traditional linear algebra, we establish a framework that enhances our understanding of near-vector spaces within this well-established domain of mathematics.\\

The paper is structured as follows: \\

In the first section, we recall the fundamental concepts and general terminology related to near-vector spaces.

In the second section, we explore the properties of the addition defined by a non-zero element of the quasi-kernel of a near-vector space, particularly about distributive elements.

In the final section, we state and prove our main theorem, which provides a new characterization of the regularity property of a near-vector space. It also offers a concluding lemma
on how our results can be utilized to determine the dimension of elements of a near-vector space (see Lemma \ref{spanfamily}).



\section{Preliminary material and notation}

In this section, we introduce the preliminary material that forms the foundation for our discussion throughout the paper, establishing the necessary notation. We begin with the concept of a scalar group.

\begin{defin}\cite[Definition 1.1]{MarquesMoore} \label{scalargroup}
A {\sf scalar group} \( F \) is a tuple \( F = (F, \cdot, 1, 0, -1) \) where \((F, \cdot, 1)\) is a monoid, \(0, -1 \in F\), satisfying the following conditions:
\begin{itemize}
    \item \(0 \cdot \alpha = 0 = \alpha \cdot 0\) for all \(\alpha \in F\);
    \item \(\{ \pm 1 \}\) is the solution set of the equation \(x^2 = 1\);
    \item \((F \setminus \{ 0 \}, \cdot, 1)\) is a group.
\end{itemize}
For all \(\alpha \in F\), we denote \(-\alpha\) as the element \((-1) \cdot \alpha\).
\end{defin}

Throughout this paper, we will use the notation \( F \) to represent the tuple \((F, \cdot, 1, 0, -1)\) when the context is clear.

In the following definition, we introduce the concept of a near-vector space over a scalar group. For the purposes of this paper, we adhere to the definition provided in \cite{MarquesMoore}.

\begin{defin}\cite[Definition 1.6]{MarquesMoore}
\label{NVS2}
A {\sf left near-vector space over a scalar group $F$} is an $F$-space \((V, F, \mu)\) where:
\begin{itemize} 
    \item $V$ is an additive abelian group,
    \item $F$ is a scalar group,
    \item \(\mu\) is a free action of $F$ on $V$,
\end{itemize} 
such that \(Q(V)\) generates \(V\) as an additive group. The set \(Q(V)\) is defined as:
\begin{equation*}
    Q(V) = \{v \in V \mid \forall \alpha, \beta \in F, \exists \gamma \in F \text{ such that } \alpha \cdot v + \beta \cdot v = \gamma \cdot v\}.
\end{equation*}
Any trivial abelian group can be endowed with a near-vector space structure through the trivial action. Such a space is referred to as a {\sf left trivial near-vector space over $F$}, denoted as \(\{0\}\).
\end{defin}

In this paper, the focus will be on left near-vector spaces, although all the results can be extended to right near-vector spaces as well. Throughout the paper, we will use "subspace" instead of "$F$-subspace" and "near-vector space" instead of "left near-vector space" unless otherwise specified. Additionally, we will use \((V, F)\) instead of \((V, F, \mu)\) when \(\mu\) is clear from the context, or simply \(V\) when \(F\) and \(\mu\) are evident.

In the following, \(F\) denotes a scalar group and \(V\) denotes a near-vector space over \(F\) unless mentioned otherwise. For a scalar group \(F\), we denote \(F^{*}\) as \(F \setminus \{0\}\), and for any subset \(W\) of \(V\) containing \(0\), we denote \(W^{*}\) as \(W \setminus \{0\}\).

A scalar basis for \( V \) is a basis such that each element of the basis is in the quasi-kernel, aligning the basis definition with that of traditional linear algebra (see \cite[Definition 2.14 (2)]{MarquesMoore}). It is proven in \cite{MarquesMoore} that every near-vector space possesses a scalar basis (see \cite[Corollary 3.6]{MarquesMoore}). 
If \(\{u_{i}\}_{i \in I}\) is a scalar basis for \(V\) over \(F\), it is straightforward to show that \(\{\lambda_{i} u_{i}\}_{i \in I}\) is also a scalar basis for \(V\) over \(F\), where \(\lambda_{i} \in F^{*}\) for all \(i \in I\). This fact will be utilized repeatedly throughout this paper.

In this paper, we aim to provide novel characterizations of regularity. Given the significance of regularity in our study, we first recall its definition.

\begin{defin}{\cite[Definition 4.7]{Andre}} \label{def4} 
Let \( u, v \in Q(V)^{*} \).
\begin{enumerate}
    \item We say that \( u \) and \( v \) are {\sf compatible} if there exists \( \lambda \in F^{*} \) such that \( u + \lambda v \in Q(V) \).
    \item We say that \( V \) is {\sf regular} if any two vectors of \( Q(V)^{*} \) are compatible.
\end{enumerate}
\end{defin}

It was proven in \cite[Theorem 2.4-18, Theorem 2.4-20]{DeBruyn} that any near-vector space can be decomposed into a direct sum of maximal regular near-vector spaces. This leads us to define the notion of a regular decomposition family for \( V \) over \( F \).

\begin{defin}
\label{decompfamily} 
Let \( v \in V \), \(\{V_{i}\}_{i \in I}\) be a family of subspaces of \( V \) and \(\{v_{i}\}_{i \in I}\) be a family of elements in \( V_{i} \) for each \( i \in I \).
\begin{enumerate}
    \item We say that \(\{V_{i}\}_{i \in I}\) is a {\sf regular decomposition family for \( V \)} if \( V = \bigoplus_{i \in I} V_{i} \) where \( V_{i} \) is a maximal regular subspace of \( V \), for all $i\in I$.
    \item We say that \(\{v_{i}\}_{i \in I}\) is the {\sf family of regular components of \( v \)} if \( v = \sum_{i \in I} v_{i} \) and \( v_{i} \in V_{i} \) for all \( i \in I \), where \(\{V_{i}\}_{i \in I}\) is a regular decomposition family for \( V \).
\end{enumerate}
\end{defin}

We conclude this section by recalling the definitions of homomorphisms and isomorphisms of near-vector spaces.

\begin{defin}{\cite[Definition 3.2]{HowMey}}\label{homo}
Let \((V_1, F_1)\) and \((V_2, F_2)\) be near-vector spaces.
\begin{enumerate}
    \item A pair \((\theta, \eta)\) is called a {\sf homomorphism} of near-vector spaces if there exists an additive homomorphism \(\theta : (V_1, +_{1}) \rightarrow (V_2, +_{2})\) and a multiplicative group isomorphism \(\eta : (F_1^{*}, \cdot_{1}) \rightarrow (F_2^{*}, \cdot_{2})\) such that \(\theta(\alpha \cdot_{1} u) = \eta(\alpha) \cdot_{2} \theta(u)\) for all \( u \in V_1 \) and \(\alpha \in F_1^{*}\).
    \item If \(\eta = \operatorname{Id}\), we say that \(\theta\) is an \( F \)-{\sf linear map}.
    \item A pair \((\theta, \eta)\) is called an {\sf isomorphism} of near-vector spaces if both \(\theta\) and \(\eta\) are bijective maps.
    \item If \(\eta = \operatorname{Id}\), we say that \(\theta\) is an {\sf \( F \)-isomorphism} of near-vector spaces.
\end{enumerate}
\end{defin}


\section{Quasi-kernel's additions and their distributive elements}

The objective of this section is to gain a deeper understanding of the additions induced by non-zero elements in the quasi-kernel, with a particular emphasis on their properties, especially those related to distributive elements. We start by recalling the definition of distributive elements in near-fields, a concept that is pivotal to our subsequent analysis.

\begin{defin}
Let $(F, +, \cdot)$ be a left near-field. An element \(\gamma \in F\) is {\sf left distributive} if, for every \(\alpha, \beta \in F\), \((\alpha + \beta)\gamma = \alpha \gamma + \beta \gamma\). We denote by \(\myfrak{d}\) the set of all left distributive elements of \(F\).
\end{defin}

Throughout the remainder of this paper, we will use the term "distributive" in place of "left distributive."

Given a near-field \( F \) and an index set \( I \), we can define a canonical near-vector space structure on the product \( F^{(I)} \), as explained in the following definition-lemma. In particular, the quasi-kernel of such a near-vector space is defined through the distributive elements of the near-field considered.

\begin{definlem}\label{F^I}\cite[Theorem 2.4-3, Theorem 2.4-7]{DeBruyn}
Let \( (F, +, \cdot) \) be a near-field and \( I \) be a non-empty index set. We define the set \( F^{(I)} \) as follows:
\begin{equation*}
    F^{(I)} = \{ (\xi_{i})_{i \in I} \mid \xi_{i} \in F, \xi_{i} \neq 0 \text{ for only a finite number of } i \in I \}.
\end{equation*}
We define addition and multiplication component-wise as follows:
\begin{equation*}
    (\xi_{i})_{i \in I} +_{F^{(I)}} (\eta_{i})_{i \in I} = (\xi_{i} + \eta_{i})_{i \in I}
\end{equation*}
and
\begin{equation*}
    \lambda (\xi_{i})_{i \in I} = (\lambda \xi_{i})_{i \in I}
\end{equation*}
for all \(\lambda, \xi_{i}, \eta_{i} \in F\). With these operations, \( F^{(I)} \) forms a near-vector space over \( F \) with quasi-kernel
\begin{equation*}
    Q(F^{(I)}) = \{\lambda (\kappa_{i})_{i \in I} \mid \lambda \in F \text{ and } \kappa_{i} \in \myfrak{d} \text{ for all } i \in I\}.
\end{equation*}
\end{definlem}

In the following definition, we introduce the addition induced by elements of the quasi-kernel.

\begin{defin} \label{+u} \cite[Definition 2.3]{Andre}
For ${u} \in Q(V)^{*}$ and $\alpha, \beta \in F$, we denote $\alpha +_{{u}} \beta$ as the unique element $\gamma$ of $F$ satisfying $\alpha {u} + \beta {u} = \gamma {u}$.
\end{defin}

It was established in \cite[Satz 2.4]{Andre} that the addition induced by non-zero elements in the quasi-kernel endows \( F \) with a near-field structure. More precisely, given \( u \in Q(V)^* \), the operation \( +_{u} \) forms a group operation on \( F \), and \((F, +_{u}, \cdot)\) is a near-field. We denote this near-field by \( F_{u} \) to emphasize that the scalar group \( F \) is endowed with a near-field structure with respect to the addition \( +_{u} \) and the original multiplication on \( F \). Given a family \((\alpha_i)_{i \in \{1, \cdots, n\}}\) of elements in \( F \), we denote by \({}^{u} \sum_{i=1}^n \alpha_i\) the sum \(\alpha_1 +_{u} \alpha_2 +_{u} \cdots +_{u} \alpha_n\).

 We know that \((\myfraku{d}, +_u, \cdot \)) is a division ring and \( F_{u} \) is a right vector space over $\myfraku{d}$. Furthermore, given \( u \in Q(V)^* \) and \( \lambda \in F^* \), the following assertions hold:

\begin{enumerate}
\item For all \(\alpha, \beta \in F\), we have \(\alpha +_{\lambda u} \beta = (\alpha \lambda +_{u} \beta \lambda) \lambda^{-1}\);
\item $\myfraku{d}$ is a subnear-field of \( F_{u} \).
\end{enumerate}

In the subsequent section, we will study the properties of the set of elements in the quasi-kernel of a near-vector space where all non-zero elements share the same addition, as defined below.

\begin{defin}
Given \( u \in Q(V)^* \), we define the set \( Q_u(V) \) as follows:
\[ 
Q_u(V) = \{ v \in Q(V)^* \mid +_u = +_v \} \cup \{0\}. 
\]
\end{defin}
\begin{rem}
    Let $V$ be a near-vector space over $F$ and $W$ be a subspace of $V$. Let $u \in W \cap Q(V)^{*}$ and $\delta \in F^{*}$. Then, we have $Q_{\delta u}(W) = Q_{\delta u}(V) \cap W$. 
\end{rem}

In the next lemma, we will examine the properties of the additions defined by non-zero elements of the quasi-kernel.

\begin{lemm}
\label{lemma:au_equals_av}
Let \(u,v \in Q(V)^{*}\). The following statements are true:
\begin{enumerate}
    \item if \(+_u = +_v\), then \(+_{\gamma u} = +_{\gamma v}\) for all \(\gamma \in F^{*}\);
    \item for all $\lambda \in \myfraku{d}^*$, we have $+_{\lambda u} = +_{u}$;
    \item $+_{\lambda u} = +_{\gamma u}$ for all $\lambda , \gamma \in F_u^{*}$ if and only if $\lambda \gamma^{-1} \in \myfraku{d}^{*}$;
    \item if $\{\delta_{j}\}_{j \in J}$ is a linearly independent set of $F_{u}$ over $\myfraku{d}$, then $+_{\delta_{i}u} \neq +_{\delta_{j}u}$ for all $i,j \in J$;
    \item given $\lambda \in F^{*}$, we have $\myfraklu{d} = \lambda \myfraku{d} \lambda^{-1}$. 
\end{enumerate}
\end{lemm}

\begin{proof}
    Let $u,v \in Q(V)^{*}$.
    \begin{enumerate}
        \item This is clear.
        \item This is clear.
        \item Suppose $\lambda \gamma^{-1} \in \myfraku{d}^{*}$. We have that $+_{\lambda u} = +_{\gamma u}$, for all $\lambda , \gamma \in F_u^{*}$, by (2). Conversely, suppose that we have $+_{\lambda u} = +_{\gamma u}$, for all $\lambda , \gamma \in F_u^{*}$. Let $\alpha , \beta \in F$. We have $\alpha +_{\lambda u} \beta = \alpha +_{\gamma u} \beta$. That is, $(\alpha \lambda +_{ u} \beta \lambda )\lambda^{-1} = (\alpha \gamma +_{ u} \beta \gamma)\gamma^{-1}$. Setting $\alpha'= \alpha \gamma$ and $\beta'= \beta \gamma$, we obtain $\alpha' +_{u}\beta'=(\alpha' \gamma^{-1} \lambda +_{u} \beta' \gamma^{-1} \lambda) (\gamma^{-1} \lambda)^{-1}$. That is, $\gamma^{-1} \lambda \in \myfraku{d}^{*}$.
        \item Suppose $+_{\delta_{i}u} = +_{\delta_{j}u}$ for some $i,j \in J$. Then by (3), $\delta_{i}\delta_{j}^{-1} \in \myfraku{d}^{*}$, which contradicts the linear independence of $\{\delta_{j}\}_{j \in J}$.
        \item Let $\lambda \in F^{*}$ and $\gamma \in \myfraku{d}$. We prove that $\lambda \gamma \lambda^{-1} \in \myfraklu{d}$. Let $\alpha, \beta \in F$.
        \begin{align*}
            (\alpha +_{\lambda u} \beta) \lambda \gamma \lambda^{-1} &= (\alpha \lambda +_u \beta \lambda) \lambda^{-1} \lambda \gamma \lambda^{-1} = (\alpha \lambda +_{u} \beta \lambda) \gamma \lambda^{-1} \\
            &= (\alpha \lambda \gamma +_{u} \beta \lambda \gamma) \lambda^{-1} = (\alpha \lambda \gamma \lambda^{-1} \lambda +_{u} \beta \lambda \gamma \lambda^{-1} \lambda) \lambda^{-1} \\
            &= \alpha (\lambda \gamma \lambda^{-1}) +_{\lambda u} \beta (\lambda \gamma \lambda^{-1}).
        \end{align*}
        Therefore, $\lambda \gamma \lambda^{-1} \in \myfraklu{d}$. 

        Let $\gamma \in \myfraklu{d}$. We prove that $\lambda^{-1} \gamma \lambda \in \myfraku{d}$. 
        \begin{align*}
            (\alpha +_{u} \beta) \lambda^{-1} \gamma \lambda &= (\alpha \lambda^{-1} \lambda +_{u} \beta \lambda^{-1} \lambda) \lambda^{-1} \gamma \lambda = (\alpha \lambda^{-1} +_{\lambda u} \beta \lambda^{-1}) \gamma \lambda \\
            &= (\alpha \lambda^{-1} \gamma +_{\lambda u} \beta \lambda^{-1} \gamma) \lambda = (\alpha \lambda^{-1} \gamma \lambda +_{u} \beta \lambda^{-1} \gamma \lambda) \lambda^{-1} \lambda \\
            &= \alpha (\lambda^{-1} \gamma \lambda) +_{u} \beta (\lambda^{-1} \gamma \lambda).
        \end{align*}
        Thus, $\lambda^{-1} \gamma \lambda \in \myfraku{d}$.
    \end{enumerate}
\end{proof}


\begin{rem}
Let $V$ be a regular near-vector space over $F$ and $u, v \in Q(V)^{*}$. We know by \cite[Lemma 2.4-12 and Theorem 2.4-15]{DeBruyn} that there exists $\lambda \in F^*$ such that $+_{v} = +_{\lambda u}$. Then we have $F_v= F_{\lambda u}$ and $\myfrakv{d} = \lambda \myfraku{d} \lambda^{-1}$. 
\end{rem}

We define the set of additions of $V$ as follows. The map in $(2)$ of the following definition is well defined by Lemma \ref{lemma:au_equals_av}.

\begin{defin}
\begin{enumerate}
    \item We define the set of additions in $V$, denoted as $\mathbf{A}_V$, to be 
    \[
    \mathbf{A}_V = \{ +_u \mid u \in Q(V)^{*} \}.
    \]
    \item For each $u \in Q(V)^{*}$, we define the map $\mathbf{a}_u: {F_u^{*}}/{\myfraku{d}^{*}} \rightarrow \mathbf{A}_V$ to be the map sending $[\alpha]_{u,\mathfrak{d}}$ to $+_{\alpha u}$, where $[\alpha]_{u,\mathfrak{d}}$ is the class of $\alpha \in F_u^{*}$ in the quotient ${F_u^{*}}/{\myfraku{d}^{*}}$.
\end{enumerate}
\end{defin}

We introduce the following notation for the canonical near-vector space $F^{(I)}$.

\begin{defin}
    Let $F$ be a near-field and $F^{(I)}$ be defined as in Definition \ref{F^I}. We denote $\mathcal{K}(F^{(I)})$ to be the set $\{(\alpha_{i})_{i \in I} \mid \alpha_{i} \in \myfraku{d} \text{ for all } i \in I \text{ and } \alpha_{i} = 0 \text{ for almost all } i \in I \}.$
\end{defin}

The next lemma explains the relevance of the previous definition.

\begin{lemm}
\label{Lem1}
    Let $u \in Q(V)^{*}$ and $\delta \in F^{*}$. Suppose that there exists an $F$-isomorphism $\phi: V \rightarrow {F_u}^{(I)}$, where ${F_u}^{(I)}$ is defined as in Definition \ref{F^I}. Then: 
    \begin{enumerate}
    \item $+_{\delta u}= +_{\delta k}$, for all $k\in \mathcal{K}(F_u^{(I)})^{*}$;
    \item $\phi ( Q_{\delta u}(V))= \{v \in Q(F_u^{(I)})^{*} \mid +_{v} = +_{\delta u} \}\cup \{0\}  = \{\delta k \mid k \in \mathcal{K}(F_u^{(I)}) \}$.
    \end{enumerate}
\end{lemm}

\begin{proof}
Let $u \in Q(V)^{*}, \delta \in F^{*}$ and $\phi:V \rightarrow F_{u}^{(I)}$ be an $F$-isomorphism.
\begin{enumerate}
    \item Let $k \in \mathcal{K}(F^{(I)}_{u})$. Then, $k = (k_{i})_{i \in I}$ where $k_{i} \in \myfraku{d}$ for all $i \in I$. Let $\alpha, \beta \in F$. Then, by definition,
    \begin{align*}
        \alpha \delta k + \beta \delta k &= (\alpha \delta k_{i}+_{u}\beta \delta k_{i})_{i \in I} = ((\alpha\delta +_{u}\beta \delta)k_{i})_{i \in I} \\
        &= ((\alpha +_{\delta u} \beta)\delta k_{i})_{i \in I} = (\alpha +_{\delta u} \beta) \delta k
    \end{align*}
    and so $+_{\delta k} = +_{\delta u}$.
    \item We first show that $\phi(Q_{\delta u}(V)) = \{v \in Q(F_{u}^{(I)})^{*} \mid +_{v} = +_{\delta u}\}\cup \{0\}$. Let $v \in Q_{\delta u}(V)$. Then $+_{v} = +_{\delta u}$. Let $\alpha, \beta \in F$. Then, $\alpha v + \beta v = (\alpha +_{\delta u}\beta)v$. If we then apply $\phi$, we have on the left-hand side that $\phi (\alpha v+ \beta v) = \alpha \phi(v)+\beta \phi(v)$ and on the right-hand side we have $\phi((\alpha+_{\delta u }\beta)v) = (\alpha+_{\delta u }\beta) \phi(v)$ and so $+_{\phi(v)} = +_{\delta u}$, which shows the forward inclusion. Similarly, we get the reverse inclusion by applying $\phi^{-1}$ to an element in $\{v \in Q(F_{u}^{(I)})^{*} \mid +_{v} = +_{\delta u}\}\cup \{0\}$. Next, we show that $$\{v \in Q(F_u^{(I)})^{*} \mid +_{v} = +_{\delta u} \}\cup \{0\}  = \{\delta k \mid k \in \mathcal{K}(F_u^{(I)}) \}.$$ Let $l \in  Q(F_{u}^{(I)})^{*}$ such that $+_{l} = +_{\delta u}$. Then, $l = \lambda k$ where $\lambda \in F^{*}$ and $k \in \mathcal{K}(F^{(I)}_{u})$ by Lemma \ref{F^I}. Hence, $+_{\lambda k} = +_{\delta u}$. By (1), we have $+_{\lambda k} = +_{\delta k}$. By Lemma \ref{lemma:au_equals_av} (3), we have $\lambda \gamma^{-1} \in \myfrak{d}^{*}$ and so $\lambda = \gamma d$ for some $d \in \myfrak{d}^{*}$. Now, we see that $v = \gamma d k = \gamma k'$ where $k'=d k \in \mathcal{K}(F_{u}^{(I)})$ which shows that $v \in \{\delta k \mid k \in \mathcal{K}(F_{u}^{(I)})\}$. The reverse inclusion follows directly from (1).
\end{enumerate}
\end{proof}

 \section{Distributive decomposition of a regular near-vector space.}
By \cite[Theorem 4.13]{Andre} or \cite[Theorem 2.4-17]{DeBruyn}, we know that any near-vector space decomposes into a direct sum of maximal regular near-vector spaces. In this section, we prove that any regular near-vector space decomposes into a specific direct sum determined by the distributive elements of $F$. We begin by proving that the set of all non-zero elements in the quasi-kernel of $V$ sharing the same addition, together with the zero element, forms a vector space over the distributive elements with respect to this addition.

\begin{lemm} \label{Vu}
Let $u \in Q(V)^{*}$ and $\lambda \in F^{*}$. The following statements are true: 
\begin{enumerate} 
    \item $Q_{u}(V)$ is a vector space over $\myfraku{d}$;
    \item $Q_{u}(V)$ is a vector space over $\myfraklu{d}$ with respect to the addition induced by $V$ and the scalar multiplication defined by $\alpha \smallstar v = ( \lambda^{-1} \alpha \lambda) \cdot v$, for all $\alpha \in \myfraklu{d}$ and $v \in Q_u(V)$;
    \item the vector space $Q_u(V)$ over $\myfraku{d}$ is isomorphic, as near-vector spaces, to the vector space $Q_u(V)$ over $\myfraklu{d}$;
    \item the vector space $Q_u(V)$ over $\myfraku{d}$ is isomorphic, as near-vector spaces, to the vector space $Q_{\lambda u}(V)$ over $\myfraklu{d}$.
\end{enumerate}
\end{lemm}

\begin{proof}
Let $u \in Q(V)^{*}$ and $\lambda \in F^{*}$. By \cite[Lemma 2.4-5]{DeBruyn}, we know that $\myfraku{d}$ is a division ring.
\begin{enumerate}
    \item $Q_u(V)$ is stable under multiplication by elements of $\myfraku{d}$ by Lemma \ref{lemma:au_equals_av}, (2). By the proof of \cite[Lemma 1.12]{MarquesMoore}, we know that $Q_u(V)$ is stable under addition. We obtain the distributive axioms by definition of $\myfraku{d}$.
    \item We have that $\myfraku{d} = \lambda^{-1} \myfraklu{d} \lambda$ by Lemma \ref{lemma:au_equals_av} (5). It follows that $\smallstar$ is an action of $\myfraklu{d}$ on $Q_u(V)$, using the fact that $\cdot$ is an action on $Q_{u}(V)$. The left distributivity follows from the left distributivity of $\cdot$. Let $\alpha, \beta \in \myfraklu{d}^{*}$ and $v \in V_{u}$. Then, we have 
    \[
    \begin{array}{lll} 
    \alpha \smallstar v + \beta \smallstar v &=& (\lambda^{-1} \alpha \lambda ) \cdot v + (\lambda^{-1} \beta  \lambda ) \cdot v = (\lambda^{-1} \alpha \lambda +_u \lambda^{-1} \beta  \lambda ) \cdot v \\ 
    &=& \lambda^{-1} ( \alpha \lambda +_u \beta  \lambda ) \lambda^{-1} \lambda \cdot v = \lambda^{-1} ( \alpha +_{\lambda u} \beta) \lambda \cdot v \\ 
    &=& (\alpha +_{\lambda u} \beta) \smallstar v
    \end{array}
    \]
    which proves the distributive property on the right. Thus, $Q_{u}(V)$ is a vector space over $\myfraklu{d}$.
    \item We define the map $\eta: \myfraku{d} \rightarrow \myfraklu{d}$ by sending $\alpha$ to $\lambda \alpha \lambda^{-1}$. It is clear that $\eta$ is a monoid isomorphism from $(\myfraku{d}, \cdot)$ to $(\myfraklu{d}, \cdot)$. Thus, $(\operatorname{Id}, \eta)$ is an isomorphism of near-vector spaces from $((Q_u(V), +, \cdot), \myfraku{d})$ to $((Q_u(V), +, \smallstar), \myfraklu{d})$.
    \item We define the map $\theta: (Q_{u}(V), +) \rightarrow (Q_{\lambda u}(V), +)$ by sending $v$ to $\lambda v$ and we define the map $\eta: \myfraku{d} \rightarrow \myfraklu{d}$ by sending $\alpha$ to $\lambda \alpha \lambda^{-1}$. As mentioned above, $\eta$ is a monoid isomorphism. Furthermore, it is clear that $\theta$ is a group isomorphism. Moreover, $\theta(\alpha \cdot u) = \lambda \alpha u = \lambda \alpha \lambda^{-1} \lambda u = \eta(\alpha) \cdot \theta(u)$. Thus, $(\theta, \eta)$ is an isomorphism of near-vector spaces from $((Q_{u}(V), +, \cdot), \myfraku{d})$ to $((Q_{\lambda u}(V), +, \smallstar), \myfraklu{d})$.
\end{enumerate}
\end{proof}

\begin{corollary}
\label{deltabasis}
Let $u \in Q(V)^{*}$, $\delta\in F^*$, and $\mathcal{B}$ be a basis of $Q_{u}(V)$ over $\myfraku{d}$. Then, $\delta\mathcal{B} = \{\delta b \mid b \in \mathcal{B}\}$ is a basis of $Q_{\delta u}(V)$ over $(F_{\delta u})^*_{\mathfrak{d}}$.
\end{corollary}

\begin{proof}
Let $u \in Q(V)^{*}, \delta \in F^{*}$ and $\mathcal{B}$ be a basis of $Q_{u}(V)$ over $\myfraku{d}$. By Lemma \ref{lemma:au_equals_av} (1), $\delta\mathcal{B} \subseteq Q_{\delta u}(V)$. Let $(\lambda_{b})_{b\in \mathcal{B}} \in \mathcal{K} (F^{(\mathcal{B})})$ such that $\sum_{b \in \mathcal{B}}\lambda_{b}\delta b = 0$. By Lemma \ref{lemma:au_equals_av} (5), since $\lambda_b \in \myfrakdu{d}$, there is $\gamma_{b} \in \myfraku{d}$ such that $\lambda_{b} = \delta \gamma_{b}\delta^{-1}$, for all $b \in \mathcal{B}$. Then,
\[
0 = \sum_{b \in \mathcal{B}}\lambda_{b}\delta b = \sum_{b \in \mathcal{B}}\delta \gamma_{b}\delta^{-1}\delta b = \sum_{b \in \mathcal{B}}\delta \gamma_{b}b = \delta \sum_{b \in \mathcal{B}}\gamma_{b}b.
\]
Since $\sum_{b \in \mathcal{B}}\gamma_{b}b = 0$ implies that $\gamma_{b} = 0$ for all $b \in \mathcal{B}$ (because $\mathcal{B}$ is a basis), it follows that $\lambda_{b} = 0$ for all $b \in \mathcal{B}$. Thus, $\delta \mathcal{B}$ is linearly independent.

Now, suppose $v \in Q_{\delta u}(V)$. By Lemma \ref{lemma:au_equals_av} (1), $\delta^{-1}v \in Q_{u}(V)$. Therefore, $v = \delta \sum_{b \in \mathcal{B}}\lambda_{b}b$ where $\lambda_{b} \in \myfraku{d}$ for all $b \in \mathcal{B}$. Note that $\delta \lambda_{b}\delta^{-1} \in (F_{\delta u})_{\mathfrak{d}}$ by Lemma \ref{lemma:au_equals_av} (5), and so we can write $v$ as $v = \sum_{b \in \mathcal{B}}\delta \lambda_{b}\delta^{-1}\delta b$. Thus, $\delta \mathcal{B}$ generates $Q_{\delta u}(V)$.
\end{proof}

The following result is the main finding of this paper, characterizing regularity in terms of the distributive decomposition of a regular near-vector space. It also elucidates the structure of the quasi-kernel along a full set of representatives of the multiplicative quotient of a non-zero element of the near-field with respect to a non-zero element of the quasi-kernel and its non-zero distributive elements.

\begin{theo}
\label{thm1}
Suppose that $V$ is a non-trivial near-vector space. Then the following assertions are equivalent:
\begin{enumerate}
    \item $V$ is regular;
    \item $Q(V) = \{ \lambda v \mid \lambda \in F, v \in Q_{u}(V) \}$ for any $u \in Q(V)^{*}$;
    \item $Q(V)^{*} = \bigcupdot_{s \in \mathcal{S}} Q_{su}^{*}(V)$, where $\mathcal{S}$ is a full set of representatives of the quotient $F_{u}^{*} / \myfraku{d}^{*}$ for some $u \in Q(V)^{*}$;
    \item $\mathbf{a}_u$ is a bijection, for some $u \in Q(V)^{*}$;
    \item $((F, +_v), (F, \cdot))$ and $((F, +_w), (F, \cdot))$ are $F$-isomorphic as left near-vector spaces for all $v, w \in Q(V)$;
    \item $Q(V) = F Q_u(V)$;
    \item $V \simeq F_u^{(I)}$, for some $u \in Q(V)^{*}$ and some set $I$;
    \item $V = \bigoplus_{j \in J} Q_{\delta_{j} u}(V)$, where $\{\delta_{j}\}_{j \in J}$ is some basis for $F$ over $\myfraku{d}$ for some $u \in Q(V)^{*}$. Moreover, $+_{\delta_{i} u} \neq +_{\delta_{j} u}$ for all $i, j \in J$;
    \item any basis $\mathcal{B}$ of $Q_u(V)$ as a $\myfraku{d}$ vector space is a scalar basis of $V$ as a near-vector space over $F$;
    \item there is a scalar basis $\mathcal{B}$ for $V$ in $Q(V)^{*}$ over $F$ such that $+_{b_1} = +_{b_2}$ for all $b_1, b_2 \in \mathcal{B}$.
\end{enumerate}
\end{theo}

\begin{proof} 
$(1) \Leftrightarrow (2)$ This follows from \cite[Satz 5.1]{Andre} or \cite[Theorem 2.5-1]{DeBruyn}. \\
$(2) \Rightarrow (3)$ Let $v \in Q(V)^{*}$. By (2), there exists $\lambda \in F^{*}$ such that $+_{v} = +_{\lambda u}$. Hence, there is $s \in \mathcal{S}$ such that $\lambda = \alpha s$ with $\alpha \in \myfraku{d}^{*}$. By Lemma \ref{lemma:au_equals_av} (2), $+_{\alpha s u} = +_{su}$, and so $+_{v} = +_{su}$. Therefore, $v \in \bigcup_{s \in \mathcal{S}} Q_{su}^{*}(V)$. Since $\mathcal{S}$ is a full set of representatives of the quotient $F_{u}^{*} / \myfraku{d}^{*}$, we conclude that $Q(V)^{*} = \bigcupdot_{s \in \mathcal{S}} Q_{su}^{*}(V)$. \\
$(3) \Rightarrow (2)$ This is clear. \\
$(2) \Leftrightarrow (4)$ This follows from Lemma \ref{lemma:au_equals_av} (2). \\
$(2) \Rightarrow (5)$ Let $v, w \in Q(V)^{*}$. By (2), there exists $\lambda \in F^{*}$ such that $w = \lambda v$. Consider the map $\theta: F_{v} \rightarrow F_{w}$ sending $\alpha$ to $\alpha \lambda^{-1}$. Let $\alpha, \beta \in F$. Then, 
\[
\theta(\alpha +_{v} \beta) = (\alpha +_{v} \beta) \lambda^{-1} 
\]
and 
\[
\theta(\alpha) +_{w} \theta(\beta) = \alpha \lambda^{-1} +_{w} \beta \lambda^{-1} = \alpha \lambda^{-1} +_{\lambda v} \beta \lambda^{-1} = (\alpha \lambda^{-1} \lambda +_{v} \beta \lambda^{-1} \lambda) \lambda^{-1} = (\alpha +_{v} \beta) \lambda^{-1}.
\]
Thus, $\theta(\alpha +_{v} \beta) = \theta(\alpha) +_{w} \theta(\beta)$. \\
$(5) \Rightarrow (2)$ Let $v, w \in Q(V)$ with $v \neq 0$. By (5), there is an $F$-isomorphism $\theta: F_{v} \rightarrow F_{w}$. Let $\alpha, \beta \in F$. Then, since $\alpha = \alpha \cdot 1$, we have $\theta(\alpha) = \theta(\alpha \cdot 1) = \alpha \theta(1)$. Moreover, $\theta(\alpha +_{v} \beta) = \theta(\alpha) +_{w} \theta(\beta)$, and so
\[
\alpha +_{v} \beta = \theta^{-1}(\theta(\alpha) +_{w} \theta(\beta)) = (\alpha \theta(1) +_{w} \beta \theta(1)) \theta(1)^{-1} = \alpha +_{\theta(1) w} \beta.
\]
This shows that $+_{v} = +_{\theta(1) w}$, and so $Q(V) \subseteq \{\lambda v \mid \lambda \in F, v \in Q_{u}(V) \}$ for any $u \in Q(V)^{*}$. The reverse inclusion is trivial. \\
$(6) \Leftrightarrow (1)$ This follows from \cite[Theorem 2.5-1]{DeBruyn}. \\
$(7) \Leftrightarrow (1)$ This follows from \cite[Theorem 2.5-2]{DeBruyn}.\\ 
 $(7) \Rightarrow (8)$  
Let $v \in V, u \in Q(V)^{*}$ and $\phi: V \rightarrow F_u^{(I)}$ be an $F$-isomorphism defined by sending $v$ to $(\lambda_{i})_{i \in I}$, where $\lambda_{i} \in F_{u}$ and $\lambda_{i} = 0$ for almost all $i \in I$. Since $F$ is a right vector space over $\myfraku{d}$ by \cite[Theorem 2.4-5]{DeBruyn}, it has a basis, say $\{\delta_{j}\}_{j \in J}$. Thus, we can write each $\lambda_{i}$ as $\lambda_{i} = \sum_{k \in K_{i}} \delta_{k} \alpha_{i_{k}}$ where $K_{i} \finsub J$ and $\alpha_{i_{k}} \in \myfraku{d}$ for all $k \in K_i$ and $i \in I$. Hence, $\phi(v)$ can be rewritten as 
\[
\phi(v) = \left(\sum_{k \in K_{i}} \delta_{k} \alpha_{i_{k}}\right)_{i \in I} = \sum_{k \in K_{i}} \delta_{k} (\alpha_{i_{k}})_{i \in I}.
\]
By Lemma \ref{Lem1}, we have $+_{\delta_{k} (\alpha_{i_{k}})_{i \in I}} = +_{\delta_{k} u}$ and $\phi(v) \in \phi(Q_{\delta_{k} u} (V))$ for all $k \in K_{i}$ and $i \in I$. This shows that $V \subseteq \sum_{j \in J} Q_{\delta_{j} u}(V)$. The reverse inclusion is trivial. By Lemma \ref{lemma:au_equals_av} (4), $+_{\delta_{i} u} \neq +_{\delta_{j} u}$ for all $i, j \in J$. 
Now we prove that the sum is direct. Suppose $\sum_{j \in J} v_{j} = 0$ where $\phi(v_{j}) = \delta_{j} (\alpha_{{j}_i})_{i \in I}$ and $(\alpha_{{j}_i})_{i \in I} \in F_u^{(I)}$ with almost all $\alpha_{{j}_i}$ being zero, for all $j \in J$. Then $\sum_{j \in J} \delta_{j} (\alpha_{j_{i}})_{i \in I} = \mathbf{0}$ implies that $\sum_{j \in J} \delta_{j} \alpha_{j_{i}} = 0$ for all $i \in I$, which implies that $\alpha_{j_{i}} = 0$ for all $i \in I$, since $\{\delta_{j}\}_{j \in J}$ is a basis for $F$. Hence, $\phi(v_{j}) = \mathbf{0}$ and so $v_{j} = 0$ for all $j \in J$.\\
$(8) \Rightarrow (7)$ Let $\mathcal{B}$ be a basis for $Q_{u}(V)$ as a vector space over $\myfraku{d}$. By Corollary \ref{deltabasis}, $\delta_{j} \mathcal{B}$ is a basis of $Q_{\delta_{j}u}(V)$ over $\myfrakdu{d}$ as a vector space for each $j \in J$. If $v \in V$, then $v = \sum_{j \in J} \sum_{b \in \mathcal{B}} \gamma_{b_{j}} \delta_{j} b$, where $\gamma_{b_{j}} \in \myfrakdu{d}$ and are almost all zero for all $b \in \mathcal{B}$ and $j \in J$.
Define the map $\psi: V \rightarrow F_{u}^{(\mathcal{B})}$ by sending $\sum_{j \in J} \sum_{b \in \mathcal{B}} \gamma_{b_{j}} \delta_{j} b$ to $\sum_{j \in J} \delta_{j} (\delta_{j}^{-1} \gamma_{b_{j}} \delta_{j})_{b \in \mathcal{B}}$. Let $v, w \in V$. We write $v = \sum_{j \in J} \sum_{b \in \mathcal{B}} \gamma_{b_{j}} \delta_{j} b$ and $w = \sum_{j \in J} \sum_{b \in \mathcal{B}} \gamma_{b_{j}}' \delta_{j} b$ as above. Then,
\begin{align*}
\psi(v + w) & = \psi\left(\sum_{j \in J} \sum_{b \in \mathcal{B}} \gamma_{b_{j}} \delta_{j} b + \sum_{j \in J} \sum_{b \in \mathcal{B}} \gamma_{b_{j}}' \delta_{j} b\right) \\
& = \psi\left(\sum_{j \in J} \delta_{j} \sum_{b \in \mathcal{B}} (\delta_{j}^{-1} \gamma_{b_{j}} \delta_{j} +_{u} \delta_{j}^{-1} \gamma_{b_{j}}' \delta_{j}) b\right) \\
& = \sum_{j \in J} \delta_{j} (\delta_{j}^{-1} \gamma_{b_{j}} \delta_{j} +_{u} \delta_{j}^{-1} \gamma_{b_{j}}' \delta_{j})_{b \in \mathcal{B}} \\
& = \sum_{j \in J} \delta_{j} (\delta_{j}^{-1} \gamma_{b_{j}} \delta_{j})_{b \in \mathcal{B}} + \sum_{j \in J} \delta_{j} (\delta_{j}^{-1} \gamma_{b_{j}}' \delta_{j})_{b \in \mathcal{B}} \\
& = \psi(v) + \psi(w).
\end{align*}
Let $\alpha \in F$. We have
\begin{align*}
    \alpha v &= \alpha \left( \sum_{j \in J} \sum_{b \in \mathcal{B}} \gamma_{b_{j}} \delta_{j} b \right)  = \alpha \left( \sum_{j \in J} \sum_{b \in \mathcal{B}} \delta_{j} (\delta_{j}^{-1} \gamma_{b_{j}} \delta_{j}) b \right) \\
    & = \alpha \sum_{b \in \mathcal{B}} \left( {}^{u \!} \!\sum_{j \in J} \delta_{j} (\delta_{j}^{-1} \gamma_{b_{j}} \delta_{j}) \right) b \\
    & = \sum_{b \in \mathcal{B}} \alpha \left( {}^{u \!} \!\sum_{j \in J} \delta_{j} (\delta_{j}^{-1} \gamma_{b_{j}} \delta_{j}) \right) b  = \sum_{b \in \mathcal{B}} \left( {}^{u \!} \!\sum_{j \in J} \delta_{j} \rho_{b_{j}} \right) b \\
    & = \sum_{j \in J} \sum_{b \in \mathcal{B}} \delta_{j} \rho_{b_{j}} b = \sum_{j \in J} \sum_{b \in \mathcal{B}} (\delta_{j} \rho_{b_{j}} \delta_{j}^{-1}) \delta_{j} b
\end{align*}
where $\alpha \left( {}^{u \!} \!\sum_{j \in J} \delta_{j} (\delta_{j}^{-1} \alpha_{b_{j}} \delta_{j}) \right) = {}^{u \!} \!\sum_{j \in J} \delta_{j} \rho_{b_{j}}$ for some $\rho_{b_{j}} \in \myfraku{d}$ for each $b \in \mathcal{B}$ and $j \in J$. Then,

\[
\psi(\alpha v) = \sum_{j \in J} \delta_{j} (\rho_{b_{j}})_{b \in \mathcal{B}}.
\]
Furthermore,
\begin{align*}
    \alpha \psi(v) &= \alpha \sum_{j \in J} \delta_{j} (\delta_{j}^{-1} \gamma_{b_{j}} \delta_{j})_{b \in \mathcal{B}} \\
    & = \alpha \left( {}^{u \!} \!\sum_{j \in J} \delta_{j} (\delta_{j}^{-1} \gamma_{b_{j}} \delta_{j}) \right)_{b \in \mathcal{B}} = \left( {}^{u \!} \!\sum_{j \in J} \delta_{j} \rho_{b_{j}} \right)_{b \in \mathcal{B}} \\
    & = \sum_{j \in J} \delta_{j} (\rho_{b_{j}})_{b \in \mathcal{B}} = \psi(\alpha v).
\end{align*}
Suppose $\sum_{j \in J} \delta_{j} (\delta_{j}^{-1} \gamma_{b_{j}} \delta_{j})_{b \in \mathcal{B}} = \mathbf{0}$. Then, $\left( {}^{u \!} \!\sum_{j \in J} \gamma_{b_{j}} \delta_{j} \right)_{b \in \mathcal{B}} = \mathbf{0}$, so that ${}^{u \!} \!\sum_{j \in J} \gamma_{b_{j}} \delta_{j} = 0$ for all $b \in \mathcal{B}$. Since $\{\delta_{j}\}_{j \in J}$ is a basis for $F$, this implies that $\gamma_{b_{j}} = 0$ for all $j \in J$ and $b \in \mathcal{B}$, which proves that $\psi$ is injective.
Let $x \in F_{u}^{(\mathcal{B})}$. Then, $x = (\lambda_{b})_{b \in \mathcal{B}}$, where $\lambda_{b} \in F$ for all $b \in \mathcal{B}$ and $\lambda_{b} \neq 0$ for a finite number of $b \in \mathcal{B}$. We have $\lambda_{b} = \sum_{j \in J} \gamma_{b_{j}} \delta_{j}$ for all $b \in \mathcal{B}$, since $\{\delta_{j}\}_{j \in J}$ is a basis for $F$. So, we see that
\begin{align*}
    x = \left( {}^{u \!} \!\sum_{j \in J} \gamma_{b_{j}} \delta_{j} \right)_{b \in \mathcal{B}} = \sum_{j \in J} (\gamma_{b_{j}} \delta_{j})_{b \in \mathcal{B}} = \sum_{j \in J} \delta_{j} (\delta_{j}^{-1} \gamma_{b_{j}} \delta_{j})_{b \in \mathcal{B}}.
\end{align*}
This proves that $\psi$ is surjective. \\
$(8) \Rightarrow (9)$ Let $\mathcal{B}$ be a basis of $Q_{u}(V)$ as a vector space over $\myfraku{d}$. By Corollary \ref{deltabasis}, $\delta_{j} \mathcal{B}$ is a basis of $Q_{\delta_{j} u}(V)$ over $\myfrakdu{d}$ as a vector space for each $j \in J$. Therefore, for $v \in V$, we have
\[
v = \sum_{j \in J} \sum_{b \in \mathcal{B}} \gamma_{b_{j}} \delta_{j} b = \sum_{b \in \mathcal{B}} \left( {}^{u \!} \!\sum_{j \in J} \gamma_{b_{j}} \delta_{j} \right) b
\]
where $\gamma_{b_{j}} \in \myfrakdu{d}$ for all $b \in \mathcal{B}$ and $j \in J$. Thus, $\mathcal{B}$ generates $V$. 
Now, suppose that $\sum_{b \in \mathcal{B}} \left( {}^{u \!} \!\sum_{j \in J} \gamma_{b_{j}} \delta_{j} \right) b = 0$, where $\gamma_{b_{j}} \in \myfrakdu{d}$ and are almost all zero for all $b \in \mathcal{B}$ and $j \in J$. This is equivalent to ${}^{u} \sum_{j \in J} \left(\sum_{b \in \mathcal{B}} \gamma_{b_{j}} \delta_{j} b \right) = 0$. Since $V = \bigoplus_{j \in J} Q_{\delta_{j} u}(V)$, we have $\sum_{b \in \mathcal{B}} \gamma_{b_{j}} \delta_{j} b = 0$ for all $j \in J$. 
Since $\mathcal{B}$ is a basis of $Q_{u}(V)$, this implies that $\gamma_{b_{j}} \delta_{j} = 0$ for all $b \in \mathcal{B}$ and $j \in J$, which in turn implies that $\gamma_{b_{j}} = 0$ for all $b \in \mathcal{B}$ and $j \in J$. This shows that $\mathcal{B}$ is a scalar basis of $V$.
 \\
$(9) \Rightarrow (10)$ This is clear.
 \\
$(10) \Rightarrow (7)$ Let $\mathcal{B}$ be a scalar basis for $V$ and define $f: V \rightarrow F_{u}^{(\mathcal{B})}$ by $f(v) = f\left(\sum_{b \in \mathcal{B}} \lambda_{b} b \right) = (\lambda_{b})_{b \in \mathcal{B}}$, where $\lambda_{b} \in F$ for all $b \in \mathcal{B}$. Choose $u \in \mathcal{B}$. Then, $+_{u} = +_{b}$ for all $b \in \mathcal{B}$. Thus,
\begin{align*}
    f(v+v') &= f\left(\sum_{b \in \mathcal{B}} \lambda_{b} b + \sum_{b \in \mathcal{B}} \lambda_{b}' b \right) = f\left(\sum_{b \in \mathcal{B}} (\lambda_{b} +_{b} \lambda_{b}') b\right) = f\left(\sum_{b \in \mathcal{B}} (\lambda_{b} +_{u} \lambda_{b}') b\right) \\
    & = (\lambda_{b} +_{u} \lambda_{b}')_{b \in \mathcal{B}} = (\lambda_{b})_{b \in \mathcal{B}} + (\lambda_{b}')_{b \in \mathcal{B}} \\
    &= f(v) + f(v').
\end{align*}
Together with \cite[Theorem 2.4-9]{DeBruyn}, this proves that $f$ is an isomorphism.

\end{proof}
We provide an example to illustrate how the theorem works in practice using a Dickson near-field.
\begin{exam}
\label{5^2}
 Consider the finite field \((\mathbb{F}_{5^2}, +, \cdot)\) and let \(\mathbf{0}, \mathbf{1}, \mathbf{2}, \mathbf{3}, \mathbf{4}\) represent the equivalence classes of \(\mathbb{F}_5\). The elements of \(\mathbb{F}_{5^2}\) are:
\[
\begin{array}{lll} GF(5^2) &:=& \{\mathbf{0}, \mathbf{1}, \mathbf{2}, \mathbf{3}, \mathbf{4}, \gamma, \mathbf{1} + \gamma, \mathbf{2} + \gamma, \mathbf{3} + \gamma, \mathbf{4} + \gamma, \mathbf{2}\gamma, \mathbf{1} + \mathbf{2}\gamma, \mathbf{2} + \mathbf{2}\gamma, \mathbf{3} + \mathbf{2}\gamma, \mathbf{4} + \\
&& \mathbf{2}\gamma, \mathbf{3}\gamma, \mathbf{1} + \mathbf{3}\gamma, \mathbf{2} + \mathbf{3}\gamma, \mathbf{3} + \mathbf{3}\gamma, \mathbf{4} + \mathbf{3}\gamma, \mathbf{4}\gamma, \mathbf{1} + \mathbf{4}\gamma, \mathbf{2} + \mathbf{4}\gamma, \mathbf{3} + \mathbf{4}\gamma, \mathbf{4} + \mathbf{4}\gamma\}
\end{array}
\]
where \(\gamma\) is a root of \(x^2 - \mathbf{3} \in \mathbb{Z}_5[x]\). The operations on \(\mathbb{F}_{5^2}\) are defined as follows:
\[
(\mathbf{a} + \mathbf{b}\gamma) + (\mathbf{c} + \mathbf{d}\gamma) = (\mathbf{a} + \mathbf{c}) + (\mathbf{b} + \mathbf{d})\gamma,
\]
and
\[
(\mathbf{a} + \mathbf{b}\gamma) \cdot (\mathbf{c} + \mathbf{d}\gamma) = [\mathbf{ac} + \mathbf{3bd}] + [\mathbf{ad} + \mathbf{bc}]\gamma,
\]
for all \(\mathbf{a}, \mathbf{b}, \mathbf{c}, \mathbf{d} \in \mathbb{Z}_5\).

The set of non-zero squares in \(\mathbb{F}_{5^2}\), denoted by \(S\), is:
\[
S = \{\mathbf{1}, \mathbf{2}, \mathbf{3}, \mathbf{4}, \mathbf{2} + \gamma, \mathbf{3} + \gamma, \mathbf{1} + \mathbf{2}\gamma, \mathbf{4} + \mathbf{2}\gamma, \mathbf{1} + \mathbf{3}\gamma, \mathbf{4} + \mathbf{3}\gamma, \mathbf{2} + \mathbf{4}\gamma, \mathbf{3} + \mathbf{4}\gamma\}.
\]

We introduce a new operation \(\circ\) on \(\mathbb{F}_{5^2}\) as follows:
\[
x \circ y := \begin{cases}
x \cdot y & \text{if } x \text{ is a square in } (\mathbb{F}_{5^2}, +, \cdot) \\
x \cdot y^5 & \text{otherwise.}
\end{cases}
\]
See Table \ref{table1} for the multiplication table. Then, \((\mathbb{F}_{5^2}, +, \circ)\) forms a near-field (see \cite{pilz}), and \((\mathbb{F}_{5^2}, \circ)\) forms a scalar group. Let \(\myfrakk{d} = \mathbb{F}_5\) for any \(u \in Q(\mathbb{F}_{5^2})^*\). Then \((\mathbb{F}_{5^2}, +, \circ)\) is a near-vector space over itself with \(Q(\mathbb{F}_{5^2}) = \mathbb{F}_{5^2}\). Additionally, \((\mathbb{F}_{5^{2}}, +, \cdot)\) is a vector space over \(\mathbb{F}_5\) with \(\{\mathbf{1}, \gamma\}\) as a basis. The set of full representatives of the quotient \((\mathbb{F}_{5^2})_u^* / \mathbb{F}_5^*\) is \(\mathcal{S} = \{\mathbf{1}\} \cup \{d + \gamma \mid d \in \mathbb{F}_5\}\).
By Lemma \ref{lemma:au_equals_av}, one can deduce that \(Q_{\mathbf{1}}(\mathbb{F}_{5^2}) = \mathbb{F}_5\) and for each \(d \in \mathbb{F}_5\), \(Q_{d + \gamma}(\mathbb{F}_{5^2}) = (d + \gamma)\mathbb{F}_5\).
Furthermore, \(Q(\mathbb{F}_{5^2})^* = \bigcupdot_{s \in \mathcal{S}} Q_s(\mathbb{F}_{5^2})^*\). We also have:
\[
\mathbb{F}_{5^2} = Q_{\mathbf{1}}(\mathbb{F}_{5^2}) \bigoplus Q_{\gamma}(\mathbb{F}_{5^2}).
\]

By defining an action by endomorphism \(\diamond\) of \(\mathbb{F}_{5^2}\) on \((\mathbb{F}_{5^2})^2\) as:
\[
\alpha \diamond (x_1, x_2) = (\alpha \circ x_1, \alpha \circ x_2),
\]
we create the regular near-vector space \(((\mathbb{F}_{5^2})^2, +,\diamond)\). Using a similar method as in \cite[Example 1.8]{MarquesMoore}, we have:
\[
Q((\mathbb{F}_{5^2})) = \{\lambda (\alpha, \beta) \mid \lambda \in \mathbb{F}_{5^2} \text{ and } \alpha, \beta \in \mathbb{F}_5\} = \{(\lambda \alpha, \lambda \beta) \mid \lambda \in \mathbb{F}_{5^2} \text{ and } \alpha, \beta \in \mathbb{F}_5\}
\]
since \(\alpha^5 = \alpha\) for all \(\alpha \in \mathbb{F}_5\). Also, \(Q((\mathbb{F}_{5^2})^2)\) is closed under addition. Thus, we can deduce that \(Q_{(\mathbf{1}, \mathbf{1})}((\mathbb{F}_{5^2})^2) = (\mathbb{F}_5)^2\) and for each \(d \in \mathbb{F}_5\),
\[
Q_{d + \gamma}((\mathbb{F}_{5^2})^2) = (d + \gamma, 0)\mathbb{F}_5 \bigoplus (0, d + \gamma)\mathbb{F}_5.
\]
Furthermore, \(Q((\mathbb{F}_{5^2})^2)^* = \bigcupdot_{s \in \mathcal{S}} Q_s((\mathbb{F}_{5^2})^2)^*\). Also, we have:
\[
(\mathbb{F}_{5^2})^2 = Q_{(\mathbf{1}, \mathbf{1})}((\mathbb{F}_{5^2})^2) \bigoplus Q_{(\gamma, \gamma)}((\mathbb{F}_{5^2})^2).
\]
\end{exam}

\renewcommand{\arraystretch}{1.5}
\begin{landscape}
\begin{table}[ht]
\centering
\scalebox{0.62}{{\begin{tabular}{cc|ccccccccccccccccccccccccc}

 & $\circ$ & $\mathbf{0}$ & $\mathbf{1}$ & $\mathbf{2}$ & $\mathbf{3}$ & $\mathbf{4}$ & $\gamma$ & $\mathbf{1}+\gamma$ & $\mathbf{2}+\gamma$ & $\mathbf{3}+\gamma$ & $\mathbf{4}+\gamma$ & $\mathbf{2}\gamma$ & $\mathbf{1}+\mathbf{2}\gamma$ & $\mathbf{2}+\mathbf{2}\gamma$ & $\mathbf{3}+\mathbf{2}\gamma$ & $\mathbf{4}+\mathbf{2}\gamma$ & $\mathbf{3}\gamma$ & $\mathbf{1}+\mathbf{3}\gamma$ & $\mathbf{2}+\mathbf{3}\gamma$ & $\mathbf{3}+\mathbf{3}\gamma$ & $\mathbf{4}+\mathbf{3}\gamma$ & $\mathbf{4}\gamma$ & $\mathbf{1}+\mathbf{4}\gamma$ & $\mathbf{2}+\mathbf{4}\gamma$ & $\mathbf{3}+\mathbf{4}\gamma$ & $\mathbf{4}+\mathbf{4}\gamma$\\ 
  \hline
 & $\mathbf{0}$ & $\mathbf{0}$ & $\mathbf{0}$ & $\mathbf{0}$ & $\mathbf{0}$ &  $\mathbf{0}$ & $\mathbf{0}$ & $\mathbf{0}$ & $\mathbf{0}$ & $\mathbf{0}$ & $\mathbf{0}$ & $\mathbf{0}$ & $\mathbf{0}$ & $\mathbf{0}$ & $\mathbf{0}$ & $\mathbf{0}$ & $\mathbf{0}$ & $\mathbf{0}$ & $\mathbf{0}$ & $\mathbf{0}$ & $\mathbf{0}$ & $\mathbf{0}$ & $\mathbf{0}$ & $\mathbf{0}$ & $\mathbf{0}$ & $\mathbf{0}$ \\ 
  &$\mathbf{1}$ & $\mathbf{0}$ & $\mathbf{1}$ & $\mathbf{2}$ & $\mathbf{3}$ & $\mathbf{4}$ & $\gamma$ & $\mathbf{1}+\gamma$ & $\mathbf{2}+\gamma$ & $\mathbf{3}+\gamma$ & $\mathbf{4}+\gamma$ & $\mathbf{2}\gamma$ & $\mathbf{1}+\mathbf{2}\gamma$ & $\mathbf{2}+\mathbf{2}\gamma$ & $\mathbf{3}+\mathbf{2}\gamma$ & $\mathbf{4}+\mathbf{2}\gamma$ & $\mathbf{3}\gamma$ & $\mathbf{1}+\mathbf{3}\gamma$ & $\mathbf{2}+\mathbf{3}\gamma$ & $\mathbf{3}+\mathbf{3}\gamma$ & $\mathbf{4}+\mathbf{3}\gamma$ & $\mathbf{4}\gamma$ & $\mathbf{1}+\mathbf{4}\gamma$ & $\mathbf{2}+\mathbf{4}\gamma$ & $\mathbf{3}+\mathbf{4}\gamma$ & $\mathbf{4}+\mathbf{4}\gamma$ \\ 
  &$\mathbf{2}$ & $\mathbf{0}$ & $\mathbf{2}$ & $\mathbf{4}$ & $\mathbf{1}$ & $\mathbf{3}$ & $\mathbf{2}\gamma$ & $\mathbf{2}+\mathbf{2}\gamma$ & $\mathbf{4}+\mathbf{2}\gamma$ & $\mathbf{1}+\mathbf{2}\gamma$ & $\mathbf{3}+\mathbf{2}\gamma$ & $\mathbf{4}\gamma$ & $\mathbf{2}+\mathbf{4}\gamma$ & $\mathbf{4}+\mathbf{4}\gamma$ & $\mathbf{1}+\mathbf{4}\gamma$ & $\mathbf{3}+\mathbf{4}\gamma$ & $\gamma$ & $\mathbf{2}+\gamma$ & $\mathbf{4}+\gamma$ & $\mathbf{1}+\gamma$ & $\mathbf{3}+\gamma$ & $\mathbf{3}\gamma$ & $\mathbf{2}+\mathbf{3}\gamma$ & $\mathbf{4}+\mathbf{3}\gamma$ & $\mathbf{1}+\mathbf{3}\gamma$ & $\mathbf{3}+\mathbf{3}\gamma$ \\ 
  &$\mathbf{3}$ & $\mathbf{0}$ & $\mathbf{3}$ & $\mathbf{1}$ & $\mathbf{4}$ & $\mathbf{2}$ & $\mathbf{3}\gamma$ & $\mathbf{3}+\mathbf{3}\gamma$ & $\mathbf{1}+\mathbf{3}\gamma$ & $\mathbf{4}+\mathbf{3}\gamma$ & $\mathbf{2}+\mathbf{3}\gamma$ & $\gamma$ & $\mathbf{3}+\gamma$ & $\mathbf{1}+\gamma$ & $\mathbf{4}+\gamma$ & $\mathbf{2}+\gamma$ & $\mathbf{4}\gamma$ & $\mathbf{3}+\mathbf{4}\gamma$ & $\mathbf{1}+\mathbf{4}\gamma$ & $\mathbf{4}+\mathbf{4}\gamma$ & $\mathbf{2}+\mathbf{4}\gamma$ & $\mathbf{2}\gamma$ & $\mathbf{3}+\mathbf{2}\gamma$ & $\mathbf{1}+\mathbf{2}\gamma$ & $\mathbf{4}+\mathbf{2}\gamma$ & $\mathbf{2}+\mathbf{2}\gamma$ \\ 
  &$\mathbf{4}$ & $\mathbf{0}$ & $\mathbf{4}$ & $\mathbf{3}$ & $\mathbf{2}$ & $\mathbf{1}$ & $\mathbf{4}\gamma$ & $\mathbf{4}+\mathbf{4}\gamma$ & $\mathbf{3}+\mathbf{4}\gamma$ & $\mathbf{2}+\mathbf{4}\gamma$ & $\mathbf{1}+\mathbf{4}\gamma$ & $\mathbf{3}\gamma$ & $\mathbf{4}+\mathbf{3}\gamma$ & $\mathbf{3}+\mathbf{3}\gamma$ & $\mathbf{2}+\mathbf{3}\gamma$ & $\mathbf{1}+\mathbf{3}\gamma$ & $\mathbf{2}\gamma$ & $\mathbf{4}+\mathbf{2}\gamma$ & $\mathbf{3}+\mathbf{2}\gamma$ & $\mathbf{2}+\mathbf{2}\gamma$ & $\mathbf{1}+\mathbf{2}\gamma$ & $\gamma$ & $\mathbf{4}+\gamma$ & $\mathbf{3}+\gamma$ & $\mathbf{2}+\gamma$ & $\mathbf{1}+\gamma$ \\ 
  &$\gamma$ & $\mathbf{0}$ & $\gamma$ & $\mathbf{2}\gamma$ & $\mathbf{3}\gamma$ & $\mathbf{4}\gamma$ & $\mathbf{2}$ & $\mathbf{2}+\gamma$ & $\mathbf{2}+\mathbf{2}\gamma$ & $\mathbf{2}+\mathbf{3}\gamma$ & $\mathbf{2}+\mathbf{4}\gamma$ & $\mathbf{4}$ & $\mathbf{4}+\gamma$ & $\mathbf{4}+\mathbf{2}\gamma$ & $\mathbf{4}+\mathbf{3}\gamma$ & $\mathbf{4}+\mathbf{4}\gamma$ & $\mathbf{1}$ & $\mathbf{1}+\gamma$ & $\mathbf{1}+\mathbf{2}\gamma$ & $\mathbf{1}+\mathbf{3}\gamma$ & $\mathbf{1}+\mathbf{4}\gamma$ & $\mathbf{3}$ & $\mathbf{3}+\gamma$ & $\mathbf{3}+\mathbf{2}\gamma$ & $\mathbf{3}+\mathbf{3}\gamma$ & $\mathbf{3}+\mathbf{4}\gamma$ \\ 
  &$\mathbf{1}+\gamma$ & $\mathbf{0}$ & $\mathbf{1}+\gamma$ & $\mathbf{2}+\mathbf{2}\gamma$ & $\mathbf{3}+\mathbf{3}\gamma$ & $\mathbf{4}+\mathbf{4}\gamma$ & $\mathbf{2}+\mathbf{4}\gamma$ & $\mathbf{3}$ & $\mathbf{4}+\gamma$ & $\mathbf{2}\gamma$ & $\mathbf{1}+\mathbf{3}\gamma$ & $\mathbf{4}+\mathbf{3}\gamma$ & $\mathbf{4}\gamma$ & $\mathbf{1}$ & $\mathbf{2}+\gamma$ & $\mathbf{3}+\mathbf{2}\gamma$ & $\mathbf{1}+\mathbf{2}\gamma$ & $\mathbf{2}+\mathbf{3}\gamma$ & $\mathbf{3}+\mathbf{4}\gamma$ & $\mathbf{4}$ & $\gamma$ & $\mathbf{3}+\gamma$ & $\mathbf{4}+\mathbf{2}\gamma$ & $\mathbf{3}\gamma$ & $\mathbf{1}+\mathbf{4}\gamma$ & $\mathbf{2}$ \\ 
  &$\mathbf{2} +\gamma$ & $\mathbf{0}$ & $\mathbf{2}+\gamma$ & $\mathbf{4}+\mathbf{2}\gamma$ & $\mathbf{1}+\mathbf{3}\gamma$ & $\mathbf{3}+\mathbf{4}\gamma$ & $\mathbf{3}+\mathbf{2}\gamma$ & $\mathbf{3}\gamma$ & $\mathbf{2}+\mathbf{4}\gamma$ & $\mathbf{4}$ & $\mathbf{1}+\gamma$ & $\mathbf{1}+\mathbf{4}\gamma$ & $\mathbf{3}$ & $\gamma$ & $\mathbf{2}+\mathbf{2}\gamma$ & $\mathbf{4}+\mathbf{3}\gamma$ & $\mathbf{4}+\gamma$ & $\mathbf{1}+\mathbf{2}\gamma$ & $\mathbf{3}+\mathbf{3}\gamma$ & $\mathbf{4}\gamma$ & $\mathbf{2}$ & $\mathbf{2}+\mathbf{3}\gamma$ & $\mathbf{4}+\mathbf{4}\gamma$ & $\mathbf{1}$ & $\mathbf{3}+\gamma$ & $\mathbf{2}\gamma$ \\ 
 &$\mathbf{3}+\gamma$ & $\mathbf{0}$ & $\mathbf{3}+\gamma$ & $\mathbf{1}+\mathbf{2}\gamma$ & $\mathbf{4}+\mathbf{3}\gamma$ & $\mathbf{2}+\mathbf{4}\gamma$ & $\mathbf{3}+\mathbf{3}\gamma$ & $\mathbf{1}+\mathbf{4}\gamma$ & $\mathbf{4}$ & $\mathbf{2}+\gamma$ & $\mathbf{2}\gamma$ & $\mathbf{1}+\gamma$ & $\mathbf{4}+\mathbf{2}\gamma$ & $\mathbf{2}+\mathbf{3}\gamma$ & $\mathbf{4}\gamma$ & $\mathbf{3}$ & $\mathbf{4}+\mathbf{4}\gamma$ & $\mathbf{2}$ & $\gamma$ & $\mathbf{3}+\mathbf{2}\gamma$ & $\mathbf{1}+\mathbf{3}\gamma$ & $\mathbf{2}+\mathbf{2}\gamma$ & $\mathbf{3}\gamma$ & $\mathbf{3}+\mathbf{4}\gamma$ & $\mathbf{1}$ & $\mathbf{4}+\gamma$ \\ 
  &$\mathbf{4} +\gamma$ & $\mathbf{0}$ & $\mathbf{4}+\gamma$ & $\mathbf{3}+\mathbf{2}\gamma$ & $\mathbf{2}+\mathbf{3}\gamma$ & $\mathbf{1}+\mathbf{4}\gamma$ & $\mathbf{2}+\gamma$ & $\mathbf{1}+\mathbf{2}\gamma$ & $\mathbf{3}\gamma$ & $\mathbf{4}+\mathbf{4}\gamma$ & $\mathbf{3}$ & $\mathbf{4}+\mathbf{2}\gamma$ & $\mathbf{3}+\mathbf{3}\gamma$ & $\mathbf{2}+\mathbf{4}\gamma$ & $\mathbf{1}$ & $\gamma$ & $\mathbf{1}+\mathbf{3}\gamma$ & $\mathbf{4}\gamma$ & $\mathbf{4}$  & $\mathbf{3}+\gamma$ & $\mathbf{2}+\mathbf{2}\gamma$ & $\mathbf{3}+\mathbf{4}\gamma$ & $\mathbf{2}$ & $\mathbf{1}+\gamma$ &  $\mathbf{2}\gamma$& $\mathbf{4}+\mathbf{3}\gamma$  \\ 
  &$\mathbf{2}\gamma$ & $\mathbf{0}$ & $\mathbf{2}\gamma$ & $\mathbf{4}\gamma$ & $\gamma$ & $\mathbf{3}\gamma$ &  $\mathbf{4}$ & $\mathbf{4}+\mathbf{2}\gamma$ & $\mathbf{4}+\mathbf{4}\gamma$ & $\mathbf{4}+\gamma$ & $\mathbf{4}+\mathbf{3}\gamma$ & $\mathbf{3}$ & $\mathbf{3}+\mathbf{2}\gamma$ & $\mathbf{3}+\mathbf{4}\gamma$ & $\mathbf{3}+\gamma$ & $\mathbf{3}+\mathbf{3}\gamma$ & $\mathbf{2}$ & $\mathbf{2}+\mathbf{2}\gamma$ & $\mathbf{2}+\mathbf{4}\gamma$ & $\mathbf{2}+\gamma$ & $\mathbf{2}+\mathbf{3}\gamma$ & $\mathbf{1}$ & $\mathbf{1}+\mathbf{2}\gamma$ & $\mathbf{1}+\mathbf{4}\gamma$ & $\mathbf{1}+\gamma$ & $\mathbf{1}+\mathbf{3}\gamma$  \\ 
  &$\mathbf{1} +\mathbf{2}\gamma$ & $\mathbf{0}$ & $\mathbf{1}+\mathbf{2}\gamma$ & $\mathbf{2}+\mathbf{4}\gamma$ & $\mathbf{3}+\gamma$ & $\mathbf{4}+\mathbf{3}\gamma$ & $\mathbf{1}+\gamma$ & $\mathbf{2}+\mathbf{3}\gamma$ & $\mathbf{3}$ & $\mathbf{4}+\mathbf{2}\gamma$ & $\mathbf{4}\gamma$ & $\mathbf{2}+\mathbf{2}\gamma$ & $\mathbf{3}+\mathbf{4}\gamma$ & $\mathbf{4}+\gamma$ & $\mathbf{3}\gamma$ & $\mathbf{1}$ & $\mathbf{3}+\mathbf{3}\gamma$ & $\mathbf{4}$ & $\mathbf{2}\gamma$ & $\mathbf{1}+\mathbf{4}\gamma$ & $\mathbf{2}+\gamma$ & $\mathbf{4}+\mathbf{4}\gamma$ & $\gamma$ & $\mathbf{1}+\mathbf{3}\gamma$ & $\mathbf{2}$ & $\mathbf{3}+\mathbf{2}\gamma$ \\ 
  &$\mathbf{2} +\mathbf{2}\gamma$ & $\mathbf{0}$ & $\mathbf{2}+\mathbf{2}\gamma$ & $\mathbf{4}+\mathbf{4}\gamma$ & $\mathbf{1}+\gamma$ & $\mathbf{3}+\mathbf{3}\gamma$ & $\mathbf{4}+\mathbf{3}\gamma$ & $\mathbf{1}$ & $\mathbf{3}+\mathbf{2}\gamma$ & $\mathbf{4}\gamma$ & $\mathbf{2}+\gamma$ & $\mathbf{3}+\gamma$ &  $\mathbf{3}\gamma$& $\mathbf{2}$ & $\mathbf{4}+\mathbf{2}\gamma$ & $\mathbf{1}+\mathbf{4}\gamma$ & $\mathbf{2}+\mathbf{4}\gamma$ & $\mathbf{4}+\gamma$ & $\mathbf{1}+\mathbf{3}\gamma$ & $\mathbf{3}$ & $\mathbf{2}\gamma$ & $\mathbf{2}+\mathbf{2}\gamma$ & $\mathbf{3}+\mathbf{4}\gamma$ & $\gamma$ & $\mathbf{2}+\mathbf{3}\gamma$ & $\mathbf{4}$ \\
  &$\mathbf{3} +\mathbf{2}\gamma$ & $\mathbf{0}$ & $\mathbf{3}+\mathbf{2}\gamma$ & $\mathbf{1}+\mathbf{4}\gamma$ & $\mathbf{4}+\gamma$ & $\mathbf{2}+\mathbf{3}\gamma$ & $\mathbf{4}+\mathbf{2}\gamma$ & $\mathbf{2}+\mathbf{4}\gamma$ & $\gamma$ & $\mathbf{3}+\mathbf{3}\gamma$ & $\mathbf{1}$ & $\mathbf{3}+\mathbf{4}\gamma$ & $\mathbf{1}+\gamma$ & $\mathbf{4}+\mathbf{3}\gamma$ & $\mathbf{2}$ & $\mathbf{2}\gamma$ & $\mathbf{2}+\gamma$ & $\mathbf{3}\gamma$ & $\mathbf{3}$ & $\mathbf{1}+\mathbf{2}\gamma$ & $\mathbf{4}+\mathbf{4}\gamma$ & $\mathbf{1}+\mathbf{3}\gamma$ & $\mathbf{4}$ & $\mathbf{2}+\mathbf{2}\gamma$ & $\mathbf{4}\gamma$ & $\mathbf{3}+\gamma$  \\
  &$\mathbf{4} +\mathbf{2}\gamma$ & $\mathbf{0}$ & $\mathbf{4}+\mathbf{2}\gamma$ & $\mathbf{3}+\mathbf{4}\gamma$ & $\mathbf{2}+\gamma$ & $\mathbf{1}+\mathbf{3}\gamma$ & $\mathbf{1}+\mathbf{4}\gamma$ & $\gamma$ & $\mathbf{4}+\mathbf{3}\gamma$ & $\mathbf{3}$ & $\mathbf{2}+\mathbf{2}\gamma$ & $\mathbf{2}+\mathbf{3}\gamma$ & $\mathbf{1}$ & $\mathbf{2}\gamma$ & $\mathbf{4}+\mathbf{4}\gamma$ & $\mathbf{3}+\gamma$ & $\mathbf{3}+\mathbf{2}\gamma$ & $\mathbf{2}+\mathbf{4}\gamma$ & $\mathbf{1}+\gamma$ & $\mathbf{3}\gamma$ & $\mathbf{4}$ & $\mathbf{4}+\gamma$ & $\mathbf{3}+\mathbf{3}\gamma$ & $\mathbf{2}$ & $\mathbf{1}+\mathbf{2}\gamma$ & $\mathbf{4}\gamma$ \\
  &$\mathbf{3}\gamma$ & $\mathbf{0}$ & $\mathbf{3}\gamma$ & $\gamma$ & $\mathbf{4}\gamma$ & $\mathbf{2}\gamma$ & $\mathbf{1}$ & $\mathbf{1}+\mathbf{3}\gamma$ & $\mathbf{1}+\gamma$ & $\mathbf{1}+\mathbf{4}\gamma$ & $\mathbf{1}+\mathbf{2}\gamma$ & $\mathbf{2}$ & $\mathbf{2}+\mathbf{3}\gamma$ & $\mathbf{2}+\gamma$ & $\mathbf{2}+\mathbf{4}\gamma$ & $\mathbf{2}+\mathbf{2}\gamma$ & $\mathbf{3}$ & $\mathbf{3}+\mathbf{3}\gamma$ & $\mathbf{3}+\gamma$ & $\mathbf{3}+\mathbf{4}\gamma$ & $\mathbf{3}+\mathbf{2}\gamma$ & $\mathbf{4}$ & $\mathbf{4}+\mathbf{3}\gamma$ & $\mathbf{4}+\gamma$ & $\mathbf{4}+\mathbf{4}\gamma$ & $\mathbf{4}+\mathbf{2}\gamma$ \\
  &$\mathbf{1} + \mathbf{3}\gamma$ & $\mathbf{0}$ & $\mathbf{1}+\mathbf{3}\gamma$ & $\mathbf{2}+\gamma$ & $\mathbf{3}+\mathbf{4}\gamma$ & $\mathbf{4}+\mathbf{2}\gamma$ & $\mathbf{4}+\gamma$ & $\mathbf{4}\gamma$ & $\mathbf{1}+\mathbf{2}\gamma$ & $\mathbf{2}$ & $\mathbf{3}+\mathbf{3}\gamma$ & $\mathbf{3}+\mathbf{2}\gamma$ & $\mathbf{4}$ & $\mathbf{3}\gamma$ & $\mathbf{1}+\gamma$ & $\mathbf{2}+\mathbf{4}\gamma$ & $\mathbf{2}+\mathbf{3}\gamma$ & $\mathbf{3}+\gamma$ & $\mathbf{4}+\mathbf{4}\gamma$ & $\mathbf{2}\gamma$ & $\mathbf{1}$ & $\mathbf{1}+\mathbf{4}\gamma$ & $\mathbf{2}+\mathbf{2}\gamma$ & $\mathbf{3}$ & $\mathbf{4}+\mathbf{3}\gamma$ & $\gamma$ \\
  &$\mathbf{2} + \mathbf{3}\gamma$ & $\mathbf{0}$ & $\mathbf{2}+\mathbf{3}\gamma$ & $\mathbf{4}+\gamma$ & $\mathbf{1}+\mathbf{4}\gamma$ & $\mathbf{3}+\mathbf{2}\gamma$ & $\mathbf{1}+\mathbf{3}\gamma$ & $\mathbf{3}+\gamma$ & $\mathbf{4}\gamma$ & $\mathbf{2}+\mathbf{2}\gamma$ & $\mathbf{4}$ & $\mathbf{2}+\gamma$ & $\mathbf{4}+\mathbf{4}\gamma$ & $\mathbf{1}+\mathbf{2}\gamma$ & $\mathbf{3}$ & $\mathbf{3}\gamma$ & $\mathbf{3}+\mathbf{4}\gamma$ & $\mathbf{2}\gamma$ & $\mathbf{2}$ & $\mathbf{4}+\mathbf{3}\gamma$ & $\mathbf{1}+\gamma$ & $\mathbf{4}+\mathbf{2}\gamma$ & $\mathbf{1}$ & $\mathbf{3}+\mathbf{3}\gamma$ & $\mathbf{2}\gamma$ & $\mathbf{2}+\mathbf{4}\gamma$  \\
  &$\mathbf{3} +\mathbf{3}\gamma$ & $\mathbf{0}$ & $\mathbf{3}+\mathbf{3}\gamma$ & $\mathbf{1}+\gamma$ & $\mathbf{4}+\mathbf{4}\gamma$ & $\mathbf{2}+\mathbf{2}\gamma$ & $\mathbf{1}+\mathbf{2}\gamma$ & $\mathbf{4}$ & $\mathbf{2}+\mathbf{3}\gamma$ & $\gamma$ & $\mathbf{3}+\mathbf{4}\gamma$ & $\mathbf{2}+\mathbf{4}\gamma$ & $\mathbf{2}\gamma$ & $\mathbf{3}$ & $\mathbf{1}+\mathbf{3}\gamma$ & $\mathbf{4}+\gamma$ & $\mathbf{3}+\gamma$ & $\mathbf{1}+\mathbf{4}\gamma$ & $\mathbf{4}+\mathbf{2}\gamma$ & $\mathbf{2}$ & $\mathbf{3}\gamma$ & $\mathbf{4}+\mathbf{3}\gamma$ & $\mathbf{2}+\gamma$ & $\mathbf{4}\gamma$ & $\mathbf{3}+\mathbf{2}\gamma$ & $\mathbf{1}$ \\
  &$\mathbf{4}+\mathbf{3}\gamma$ & $\mathbf{0}$ & $\mathbf{4}+\mathbf{3}\gamma$ & $\mathbf{3}+\gamma$ & $\mathbf{2}+\mathbf{4}$ & $\mathbf{1}+\mathbf{2}\gamma$  & $\mathbf{4}+\mathbf{4}\gamma$ & $\mathbf{3}+\mathbf{2}\gamma$ & $\mathbf{2}$ & $\mathbf{1}+\mathbf{3}\gamma$ & $\gamma$ & $\mathbf{3}+\mathbf{3}\gamma$ & $\mathbf{2}+\gamma$ & $\mathbf{1}+\mathbf{4}\gamma$ & $\mathbf{2}\gamma$ & $\mathbf{4}$ & $\mathbf{2}+\mathbf{2}\gamma$ & $\mathbf{1}$ & $\mathbf{3}\gamma$ & $\mathbf{4}+\gamma$ & $\mathbf{3}+\mathbf{4}\gamma$ & $\mathbf{1}+\gamma$ & $\mathbf{4}\gamma$ & $\mathbf{4}+\mathbf{2}\gamma$ & $\mathbf{3}$ & $\mathbf{2}+\mathbf{3}\gamma$  \\
  &$\mathbf{4}\gamma$ & $\mathbf{0}$ & $\mathbf{4}\gamma$ & $\mathbf{3}\gamma$ & $\mathbf{2}\gamma$ & $\gamma$  & $\mathbf{3}$ & $\mathbf{3}+\mathbf{4}\gamma$ & $\mathbf{3}+\mathbf{3}\gamma$ & $\mathbf{3}+\mathbf{2}\gamma$ & $\mathbf{3}+\gamma$  & $\mathbf{1}$ & $\mathbf{1}+\mathbf{4}\gamma$ & $\mathbf{1}+\mathbf{3}\gamma$ & $\mathbf{1}+\mathbf{2}\gamma$ & $\mathbf{1}+\gamma$ & $\mathbf{4}$ & $\mathbf{4}+\mathbf{4}\gamma$ & $\mathbf{4}+\mathbf{3}\gamma$ & $\mathbf{4}+\mathbf{2}\gamma$ & $\mathbf{4}+\gamma$ & $\mathbf{2}$ & $\mathbf{2}+\mathbf{4}$ & $\mathbf{2}+\mathbf{3}\gamma$ & $\mathbf{2}+\mathbf{2}\gamma$ & $\mathbf{2}+\gamma$ \\
  &$\mathbf{1} +\mathbf{4}\gamma$ & $\mathbf{0}$ & $\mathbf{1}+\mathbf{4}\gamma$ & $\mathbf{2}+\mathbf{3}\gamma$ & $\mathbf{3}+\mathbf{2}\gamma$ & $\mathbf{4}+\gamma$ & $\mathbf{3}+\mathbf{4}\gamma$ & $\mathbf{4}+\mathbf{3}\gamma$ & $\mathbf{2}\gamma$ & $\mathbf{1}+\gamma$ & $\mathbf{2}$ & $\mathbf{1}+\mathbf{3}\gamma$ & $\mathbf{2}+\mathbf{2}\gamma$ & $\mathbf{3}+\gamma$ & $\mathbf{4}$ & $\mathbf{4}\gamma$ & $\mathbf{4}+\mathbf{2}\gamma$ & $\gamma$ & $\mathbf{1}$ & $\mathbf{2}+\mathbf{4}\gamma$ & $\mathbf{3}+\mathbf{3}\gamma$ & $\mathbf{2}+\gamma$ & $\mathbf{3}$ & $\mathbf{4}+\mathbf{4}\gamma$ & $\mathbf{3}\gamma$ & $\mathbf{1}+\mathbf{2}\gamma$  \\
  &$\mathbf{2} +\mathbf{4}\gamma$ & $\mathbf{0}$ & $\mathbf{2}+\mathbf{4}\gamma$ & $\mathbf{4}+\mathbf{3}\gamma$ & $\mathbf{1}+\mathbf{4}\gamma$ & $\mathbf{3}+\gamma$ & $\mathbf{2}+\mathbf{2}\gamma$ & $\mathbf{4}+\gamma$ & $\mathbf{1}$ & $\mathbf{3}+\mathbf{4}\gamma$ & $\mathbf{3}\gamma$ & $\mathbf{4}+\mathbf{4}\gamma$ & $\mathbf{1}+\mathbf{3}\gamma$ & $\mathbf{3}+\mathbf{2}\gamma$ & $\gamma$ & $\mathbf{2}$ & $\mathbf{1}+\gamma$ & $\mathbf{3}$ & $\mathbf{4}\gamma$ & $\mathbf{2}+\mathbf{3}\gamma$ & $\mathbf{4}+\mathbf{2}\gamma$ & $\mathbf{3}+\mathbf{3}\gamma$ & $\mathbf{2}\gamma$ & $\mathbf{2}+\gamma$ & $\mathbf{4}$ & $\mathbf{1}+\mathbf{4}\gamma$  \\
  &$\mathbf{3}+\mathbf{4}\gamma$ & $\mathbf{0}$ & $\mathbf{3}+\mathbf{4}\gamma$ & $\mathbf{1}+\mathbf{3}\gamma$ & $\mathbf{4}+\mathbf{2}$ & $\mathbf{2}+\gamma$ & $\mathbf{2}+\mathbf{3}\gamma$ & $\mathbf{2}\gamma$ & $\mathbf{3}+\gamma$ & $\mathbf{1}$ & $\mathbf{4}+\mathbf{4}\gamma$ & $\mathbf{4}+\gamma$ & $\mathbf{2}$ & $\mathbf{4}\gamma$ & $\mathbf{3}+\mathbf{3}\gamma$ & $\mathbf{1}+\mathbf{2}\gamma$ & $\mathbf{1}+\mathbf{4}\gamma$ & $\mathbf{4}+\mathbf{3}\gamma$ & $\mathbf{2}+\mathbf{2}\gamma$ & $\gamma$ & $\mathbf{3}$ & $\mathbf{3}+\mathbf{2}\gamma$ & $\mathbf{1}+\gamma$ & $\mathbf{4}$ & $\mathbf{2}+\mathbf{4}\gamma$ & $\mathbf{3}\gamma$ \\
  &$\mathbf{4}+\mathbf{4}\gamma$ & $\mathbf{0}$ & $\mathbf{4}+\mathbf{4}\gamma$ & $\mathbf{3}+\mathbf{3}\gamma$ & $\mathbf{2}+\mathbf{2}\gamma$ & $\mathbf{1}+\gamma$ & $\mathbf{3}+\gamma$ & $\mathbf{2}$ & $\mathbf{1}+\mathbf{4}\gamma$ & $\mathbf{3}\gamma$ & $\mathbf{4}+\mathbf{2}\gamma$ & $\mathbf{1}+\mathbf{2}\gamma$ & $\gamma$ & $\mathbf{4}$ & $\mathbf{3}+\mathbf{4}\gamma$ & $\mathbf{2}+\mathbf{3}\gamma$ & $\mathbf{4}+\mathbf{3}\gamma$ & $\mathbf{3}+\mathbf{2}\gamma$ & $\mathbf{2}+\gamma$ & $\mathbf{1}$ & $\mathbf{4}\gamma$ & $\mathbf{2}+\mathbf{4}\gamma$ & $\mathbf{1}+\mathbf{3}\gamma$ & $\mathbf{2}\gamma$ & $\mathbf{4}+\gamma$ & $\mathbf{3}$\\ \bigskip
\end{tabular}}}
\caption{Multiplication table for the near-field $\mathbb{F}_{5^{2}}$ \label{table1}} 
\end{table} 
\end{landscape}

The following corollary arises as a direct consequence of Theorem \ref{thm1}.

\begin{cor}
\label{dimension}
Let $V$ be a regular near-vector space over $F$. Then
\begin{equation*}
    \operatorname{dim}(V) = \operatorname{dim}_{\myfraku{d}}(Q_{u}(V))
\end{equation*}
for all $u \in Q(V)^{*}$.
\end{cor}

\begin{rem}
\label{rem}
\begin{enumerate}
\item Corollary \ref{dimension} is not true without the assumption that our near-vector space is regular. Consider the scalar group $(\mathbb{R},\cdot)$, which is the underlying multiplicative group of the field $(\mathbb{R}, +, \cdot)$. We define an action by endomorphism of $\mathbb{R}$ on $\mathbb{R}^{3}$ as follows:
\begin{equation*}
    \alpha \smallstar (x, y, z) = (\alpha x, \alpha^{3} y, \alpha^{3} z)
\end{equation*}
for $\alpha \in \mathbb{R}$ and $(x, y, z) \in \mathbb{R}^{3}$. Then $(\mathbb{R}^{3}, +, \smallstar)$ becomes a near-vector space and has a scalar basis $\{(1, 0, 0), (0, 1, 0), (0, 0, 1)\}$. In this case, $Q_{(1, 0, 0)}(\mathbb{R}^{3}, +, \smallstar) = \mathbb{R} \times \{0\} \times \{0\}$, and so $\operatorname{dim}(Q_{(1, 0, 0)}(\mathbb{R}^{3}, +, \smallstar)) = 1$, whereas $\operatorname{dim}(\mathbb{R}^{3}) = 3$. 
\item When $V$ is a near-vector space that is not necessarily regular, we know by \cite[Satz 4.13]{Andre} or \cite[Theorem 2.4-17]{DeBruyn} that any near-vector space can be decomposed into a direct sum of its regular subspaces. Theorem \ref{thm1} provides an additional decomposition into distributive subspaces.
\end{enumerate}
\end{rem}

The dimension of an element in a near-vector space \(V\) is defined as the minimal number of elements in the quasi-kernel required to express the element as a sum of those elements (see \cite[Definition 3.5]{Howellspanning}). The next lemma provides the regular decomposition of the span of an element in \(V\).
\begin{lemm}
\label{spanfamily}
    Let \(v \in V\) and suppose \(\{v_i\}_{i \in I}\) is a family of regular components of \(v\), where \(v_i \in V_i\) for all \(i \in I\), and \(\{V_i\}_{i \in I}\) is a regular decomposition family for \(V\). Then \(\{\operatorname{span}(v_i)\}_{i \in I}\) is the regular decomposition family for \(\operatorname{span}(v)\). In particular, we have
    \[
    \operatorname{dim}(v) = \sum_{i \in I} \operatorname{dim}(v_i) \geq |\{v_i \mid v_i \neq 0 \text{ for all } i \in I\}|.
    \]
\end{lemm}
\begin{proof}
By \cite[Satz 4.14]{Andre} or \cite[Theorem 2.4-17]{DeBruyn}, we know that \(V\) has a regular decomposition family. Let \(v \in V\). According to \cite[Theorem 3.2]{MarquesMoore}, there exists a set \(\Theta \subseteq Q(V)\) such that \(v = \sum_{q \in \Theta} q\) and
\begin{equation}
\label{equation2}
    \operatorname{span}(v) = \bigoplus_{q \in \Theta} F q.
\end{equation}
Define an equivalence relation \(\sim\) on \(\Theta\) as follows: given \(q_1, q_2 \in \Theta\), \(q_1 \sim q_2\) if and only if there exists some \(s \in F^*\) such that \(+_{q_1} = +_{sq_2}\). If \(\mathcal{S}\) is a set of full representatives of the equivalence classes of elements in \(\Theta\) under the equivalence relation \(\sim\), then
\begin{equation}
\label{equation1}
    \Theta = \bigcupdot_{s \in \mathcal{S}} \Theta_s
\end{equation}
where \(\{\Theta_s\}_{s \in \mathcal{S}}\) are the equivalence classes.
Thus, 
\[v = \sum_{q \in \bigcupdot_{s \in \mathcal{S}} \Theta_s} q = \sum_{s \in \mathcal{S}} \sum_{q \in \Theta_s} q = \sum_{s \in \mathcal{S}} v_s\]
where \(v_s = \sum_{q \in \Theta_s} q\) for all \(s \in \mathcal{S}\). Hence, \(\{v_s\}_{s \in \mathcal{S}}\) is a family of regular components of \(v\). Clearly, \(\operatorname{span}(v) \subseteq \bigoplus_{s \in \mathcal{S}} \operatorname{span}(v_s)\), and \(\operatorname{span}(v_s) \subseteq \bigoplus_{q \in \Theta_s} F q\) for all \(s \in \mathcal{S}\). 
Then, by (\ref{equation2}) and (\ref{equation1}), we have
\begin{align*}
    \bigoplus_{s \in \mathcal{S}} \operatorname{span}(v_s) & \subseteq \bigoplus_{s \in \mathcal{S}} \bigoplus_{q \in \Theta_s} F q = \bigoplus_{q \in \bigcupdot \Theta_s} F q = \bigoplus_{q \in \Theta} F q = \operatorname{span}(v).
\end{align*}
Therefore, \(\operatorname{span}(v) = \bigoplus_{s \in \mathcal{S}} \operatorname{span}(v_s)\), which proves that \(\{\operatorname{span}(v_s)\}_{s \in \mathcal{S}}\) is the regular decomposition family for \(\operatorname{span}(v)\).
\end{proof}

Using the notation of the previous lemma, we see that in order to understand \(\operatorname{span}(v)\), it is sufficient to understand \(\operatorname{span}(v_i)\) for all \(i \in I\). This can be achieved using any of the characterizations of regularity provided in Theorem \ref{thm1}.

\end{document}